\documentclass[11pt,a4paper]{elsarticle}
\usepackage[left=2cm,right=2cm,top=1.5cm,bottom=2cm]{geometry}
\biboptions{authoryear,round}
\usepackage{graphicx} 
\usepackage{float}
\usepackage{tikz, xcolor, pgfplots}
\usepackage{makecell, pifont, dsfont} 
\usepackage{booktabs}
\usetikzlibrary{shapes.geometric, arrows.meta}
\usepackage{todonotes}
\usepackage{amsmath}
\allowdisplaybreaks
\usepackage{comment}
\usepackage{amssymb}
\usepackage{amsthm}
\usepackage{todonotes}
\usepackage{booktabs,caption, mathtools}
\usepackage{hyperref}
\hypersetup{hidelinks}
\newtheorem{thrm1}{Theorem}
\newtheorem{thrm2}{Theorem}
\newtheorem{thrm3}{Theorem}
\newtheorem{thrm4}{Theorem}
\newtheorem{thrm5}{Theorem}

\newtheorem{thrm7}{Theorem}

\newtheorem{corollary}[thrm1]{Corollary}
\newtheorem{remark}[thrm2]{Remark}
\newtheorem{proposition}[thrm3]{Proposition}
\newtheorem{theorem}[thrm4]{Theorem}
\newtheorem{definition}[thrm5]{Definition}

\newtheorem{lemma}[thrm7]{Lemma}

\DeclareMathOperator*{\argmax}{arg\,max}
\DeclareMathOperator*{\argmin}{arg\,min}
\usepackage{algorithm,subfigure}
\usepackage{algpseudocode, marvosym}

\begin{document}

\begin{frontmatter}

\title{A Framework for Eliminating Paradoxical Orders in European Day-Ahead\\[0.5em]Electricity Markets  through Mixed-Integer Linear Programming Strong Duality}

\author[1]{Zhen Wang}\corref{cor1}
\ead{zhenwa@kth.se}
\author[1]{Mohammad Reza Hesamzadeh} 
\ead{mrhesa@kth.se}
\author[2]{Shudian Zhao}
\ead{shudian@kth.se}
\author[2]{Jan Kronqvist}
\ead{jankr@kth.se}
\cortext[cor1]{Corresponding author}
\affiliation[1]{organization={KTH Royal Institute of Technology, Department of Electrical Engineering and Computer Science},
            city={SE-100 44, Stockholm},
            country={Sweden}}
\affiliation[2]{organization={KTH Royal Institute of Technology, Department of Mathematics},
            city={SE-100 44, Stockholm},
            country={Sweden}}
            
\begin{abstract}
The presence of integer variables in the European day-ahead electricity market renders the social welfare maximization problem non-convex and non-differentiable, making classical marginal pricing theoretically inconsistent. Existing pricing mechanisms often struggle to balance revenue adequacy with incentive compatibility, typically relying on discriminatory uplift payments or suffering from weak duality.
Leveraging the Augmented Lagrangian Duality (ALD) framework, which establishes strong duality for Mixed-Integer Linear Programming (MILP), this paper proposes a novel ALD pricing mechanism. We analytically prove that this mechanism is inherently incentive-compatible, aligning centralized dispatch with individual incentives without requiring side payments. Notably, we demonstrate that the ALD price signals intrinsically eliminate Paradoxically Rejected Orders (PROs) and Paradoxically Accepted Orders (PAOs) for supply orders. For the demand side, a sufficient condition is introduced to further guarantee revenue adequacy, resulting in a transparent and financially consistent settlement system.
To ensure computational tractability, we modify the Surrogate Absolute-Value Lagrangian Relaxation (SAVLR) method to efficiently compute the exact penalty coefficients and optimal Lagrangian multipliers. Numerical experiments on illustrative examples and the Nordic 12-area electricity market model confirm the superior economic properties of the ALD pricing mechanism and the tractability of the modified SAVLR algorithm.
\end{abstract}
%
\end{frontmatter}
\section{Introduction}
\label{sec:intro}
The Nominated Electricity Market Operator (NEMO) provides price signals to coordinate dispatch. While convex market models possess Walrasian equilibrium and are solvable in polynomial time \citep[Section 3.2]{ahunbay2024solving}, the introduction of integer variables, which represent indivisible physical realities like start-up costs and complex "all-or-nothing" bid structures, renders the social-welfare optimization problem non-convex. This non-convexity creates a duality gap where classical marginal prices cannot guarantee that all market orders remain revenue-adequate. Consequently, current pricing mechanisms rely on ex-post uplift payments, which are often discriminatory and lack transparency \citep{madani2018convex}.

\subsection{Literature Review}
The foundational structure of the European day-ahead energy market, encompassing diverse bid types, complex operational constraints, and welfare formulations for orders, is comprehensively detailed by \citet{martin2014strict,chatzigiannis2016european}. Research works have consistently sought to make the model both comprehensive and computationally tractable. For instance, \citet{vlachos2016comparison} introduced flexible hourly profile orders and compared a computationally efficient Linear Programming (LP) model with a flexible yet computationally demanding Mixed Complementarity Problem (MCP) formulation. To enhance compatibility with state-of-the-art solvers, \citet{sleisz2019new} proposed new MILP formulations for minimum income, scheduled stop, and load gradient conditions. Furthermore, \citet{ilea2018european} utilized a Mixed-Integer Quadratically Constrained Programming (MIQCP) model to analyze the correlation between Paradoxically Accepted Orders (PAOs) and Paradoxically Rejected Orders (PROs), while \citet{kuang2019pricing} incorporated convex quadratic deliverability costs for renewable integration, though their pricing approach can inadvertently generate negative prices.

\subsubsection{Pricing Mechanisms and the Uplift Problem}
Due to inherent non-convexity, classical pricing mechanisms often result in negative profits for market orders, necessitating frameworks that ensure the market orders remain ``in-the-money.''
\begin{itemize}
    \item \textbf{IP and modified IP Pricing:} The Integer Programming (IP) pricing mechanism \citep{o2005efficient} solves the centralized problem and fixes binary variables to their optimal values, utilizing dual variables for marginal pricing and uplift payments. The modified IP \citep{bjorndal2008equilibrium} improves upon this by fixing specific variables and generating valid support inequalities to provide a non-volatile, non-discriminatory price function for pure integer programming problems.
    \item \textbf{Alternative Rules:} Other methods include uniform pricing for buyers and sellers separately to minimize compensation costs \citep{toczylowski2009new}, and the recently proposed ``markup'' mechanism by \citet{ahunbay2024solving}, where sellers receive a price signal $p^\alpha$ and buyers pay by another price $(1+\alpha)p^\alpha $. Although this mechanism reduces uplift payment, it requires specific rounding techniques. More recently, the IP and Min-MWP (make-whole payment) rule \citep{ahunbay2025pricing} was introduced to minimize side payments while maintaining congestion signals. A critical comparison of mechanisms up to 2016 is provided by \citet{liberopoulos2016critical} and extended by \citet{azizan2020optimal} through an Equilibrium-Constrained (EC) pricing mechanism.
\end{itemize}

\subsubsection{Convex Hull Pricing (CHP) and Long-Run Impacts}
Convex Hull Pricing (CHP) is widely utilized in US markets to minimize uplift while maintaining uniform prices \citep{gribik2007market}. Analyses by \citet{stevens2024some} and long-run market studies by \citet{byers2023long} suggest that CHP performs optimally regarding social welfare and producer surplus transfer. Technically, \citet{hua2016convex} proved that formulating the convex hull is equivalent to solving a Lagrangian relaxation problem. While subgradient and cutting-plane methods \citep{wang2013subgradient} can compute an approximation price, the exact convex hull formulation remains computationally difficult. Specialized solutions exist for specific generators, such as those in the Midcontinent Independent System Operator (MISO) system \citep{yu2020extended}, or for thermal units \citep{knueven2022computationally}. Although Dantzig-Wolfe decomposition \citep{andrianesis2021computation} and extreme point methods \citep{yang2019unified} can find exact convex-hull prices, they face heavy computational burdens in large-scale multi-period models.

\subsubsection{Paradoxical Orders and Investment Incentives}
The long-term health of the electricity market depends heavily on pricing stability and reliable investment incentives. Research comparing capacity expansion and unit commitment models by \citet{mays2021investment} suggests that pricing mechanisms significantly affect the equilibrium capacity mix. In the European context, avoiding PAOs is essential; however, the required revenue adequacy constraints are typically nonlinear. \citet{madani2015computationally,madani2017mip} developed efficient MIP formulations and linearized minimum income conditions without introducing auxiliary continuous or binary variables. Nevertheless, these methods struggle computationally when block orders dominate the market clearing. Later work by \citet{madani2018revisiting} generalized these developments as ``Minimum Profit'' (MP) conditions and proposed a decomposition procedure utilizing Benders cuts to handle complex ramping constraints.
\par
Despite these extensive advances, existing pricing mechanisms for non-convex markets fundamentally rely on weak duality, or ex-post uplift payments. The properties of these rules are linked to the optimality gap; for example, \citet{byers2022economic} observed that as the duality gap narrows, consumer surplus typically increases. Furthermore, from an investment perspective, linear pricing rules have generally been shown to be more efficient than nonlinear pricing mechanisms \citep{herrero2015electricity}. 
\par
To address these deficiencies, this paper proposes a transparent piecewise-linear pricing mechanism for the European day-ahead market.
Exploiting MILP strong duality under the ALD framework \citep{feizollahi2017exact}, the proposed
framework intrinsically eliminates paradoxical orders (a sufficient condition is needed for demand-side PAO elimination) without the need for
explicit revenue adequacy constraints within the optimization model.
\par
\textbf{The contribution of this paper:}
This paper proposes a novel pricing mechanism based on the Augmented Lagrangian Duality (ALD) approach for the European day-ahead market. Our contributions are:
\begin{itemize}
    \item \textbf{(C1) Mechanism Formulation and Economic Properties}: We propose a piecewise-linear pricing rule (ALD pricing) that eliminates PROs entirely and prevents PAOs for supply orders. If demand bids exceed supply offer prices, PAOs among demand orders are also eliminated. This mechanism is incentive-compatible and requires zero uplift payments due to transparent price signals. This mechanism generalizes classical Locational Marginal Pricing (LMP) by extending it from an LP to an MILP model, in which the penalty coefficient simply vanishes. Furthermore, this pricing mechanism can also provide congestion prices for Transmission System Operators (TSOs).
    \item \textbf{(C2) Algorithmic Innovation}: We modify the Surrogate Absolute-Value Lagrangian Relaxation (SAVLR) method \citep{bragin2018scalable} to close the MILP duality gap. Using the decomposable structure of the welfare maximization problem and motivated by recent results regarding the polynomial-time computation of exact penalty coefficients \citep{lefebvre2024exact}, we ensure that the proposed ALD mechanism is computationally tractable. 
\item \textbf{(C3) Comprehensive Comparative Analysis and Numerical Validation}: We extensively validate the proposed ALD mechanism to bridge theory and practice. First, using an established benchmark \citep{schiro2015convex}, we compare our mechanism with the existing schemes across the four dimensions: individual revenue adequacy, incentive compatibility, economic efficiency, and transparency. The analysis demonstrates that the ALD mechanism outperforms traditional benchmarks by successfully reconciling these often-conflicting economic properties (validating C1). Furthermore, through a stylized computation on the Nordic 12-area electricity market, we confirm not only the mechanism's welfare equivalence but also the computational tractability of our proposed algorithm at a practical scale (validating C2).
\end{itemize}
\begin{table}[htbp]
  \centering
  \caption{Summary of major pricing schemes and their properties.}
  \label{tab:pricing_comparison}
\scriptsize
  \begin{tabular}{cl|ccccc}
    \toprule
    \makecell[l]{\textbf{Scheme}\\ \textbf{/property}} &
    \makecell[l]{\textbf{Price form} \\ ${\bf p_i(q_i) = }$} & \makecell[l]{\textbf{Proposed for} \\ ${\bf c_i(q_i) = }$}& 
    \makecell[l]{\textbf{Market}\\ \textbf{clearing}}& \makecell{\textbf{Revenue}\\\textbf{Adequate}} & 
    \makecell{\textbf{Supports}\\ \textbf{competitive}\\\textbf{equilibrium}}& \makecell{\textbf{Economically}\\\textbf{Efficient}} \\
    \midrule
    \makecell{\textbf{Shadow}\\\textbf{pricing}} & ${\bf \lambda q_i}$ & {\bf Convex }& \ding{51} & \ding{51} & \ding{51} & \ding{51} \\
    \addlinespace[0.5em]
    \textbf{IP} & ${\bf \lambda q_i +u_i \mathds{1}_{q_i>0}}$&{\bf \makecell{Start-up\\ plus linear} }& \ding{51} & \ding{51} &\ding{51} & \ding{51} \\
    \addlinespace[0.5em]
    \textbf{CHP} & ${\bf \lambda q_i +u_i \mathds{1}_{q_i= q_i^*} }$ & {\bf \makecell{Start-up\\ plus linear}} & \ding{51} & \ding{51} & \ding{51} & \ding{55}\\
    \addlinespace[0.5em]
       \textbf{SLR} & $\lambda q_i$ & {\bf \makecell{Start-up\\ plus linear}} & \ding{51} & \ding{51} & \ding{55} & \ding{55}\\
       \addlinespace[0.5em]
          \textbf{PD} & {\bf $\lambda q_i$} & {\bf \makecell{Start-up\\ plus linear}} & \ding{51} & \ding{51} & \ding{55} & \ding{55}\\
              \addlinespace[0.5em]
     \textbf{EC}  & {\bf User-specified} &{\bf General}& \ding{51} & \ding{51} & \ding{51} & \ding{51} \\
    \addlinespace[1.0em]
     \makecell{\textbf{ALD}\\\textbf{(Proposed)}}  &  $\begin{smallmatrix} {\bf \lambda q_i+ \rho q_i^*} \\  {\bf -\rho |q_i - q_i^*|}\end{smallmatrix}$& {\bf \makecell{Start-up\\ plus linear}}& \ding{51} & \ding{51} & \ding{51} & \ding{51} \\
    \bottomrule
  \end{tabular}
\end{table}
According to the discussion above, we update \citep[p. 481, Table 1]{azizan2020optimal} by adding the ALD pricing mechanism as in Table \ref{tab:pricing_comparison}.
In addition, the proposed ALD pricing mechanism formulated in Table \ref{tab:pricing_comparison} has a formula: $\lambda q_i+ \rho q_i^*   -\rho |q_i - q_i^*|$, where $\lambda$ and $\rho$ are the ALD pricing signals to be introduced in the following sections, and $q_i^*$ is the optimal solution of the welfare-maximization problem. It is clear that the proposed ALD pricing mechanism is piecewise linear.
 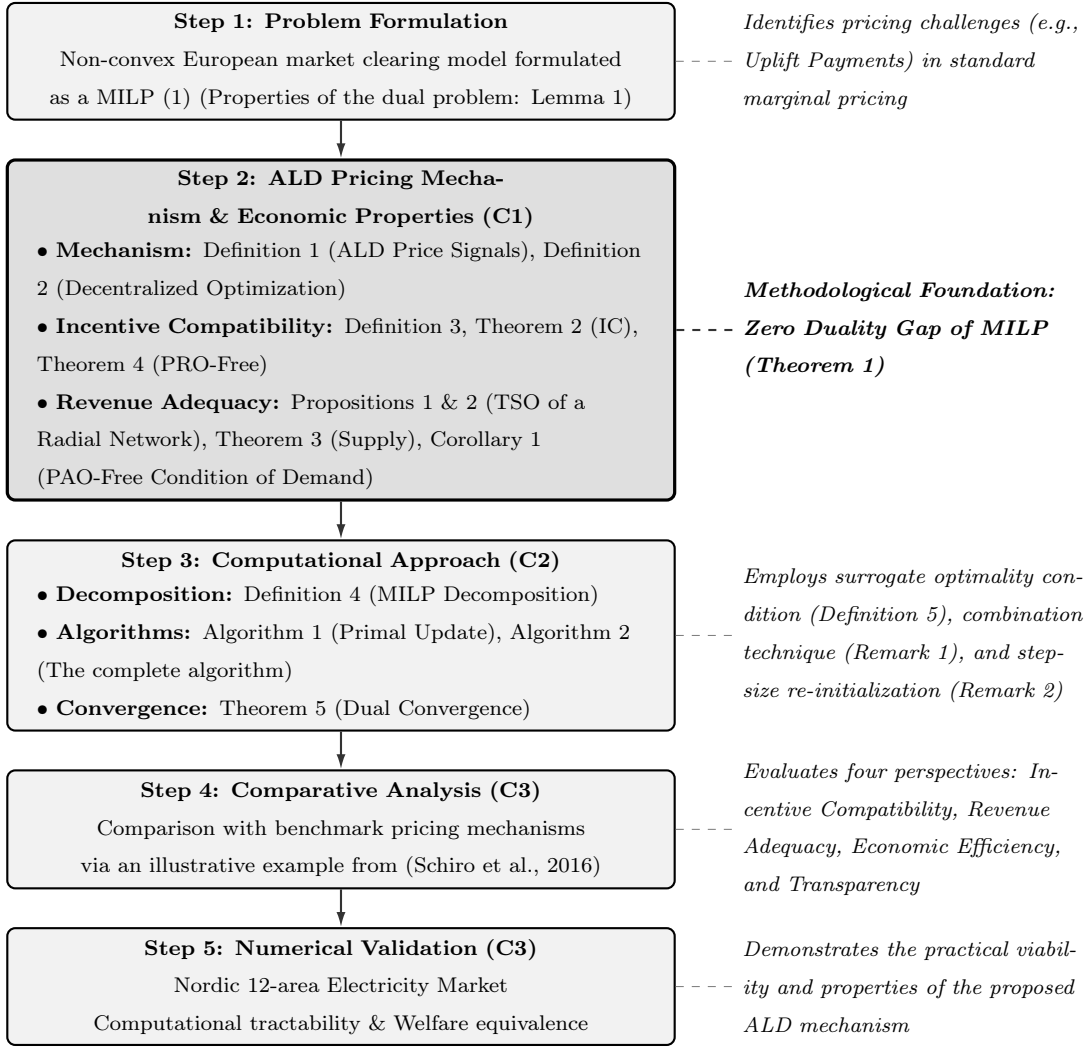
\begin{figure}[H]
\centering
\begin{tikzpicture}[
    box/.style={rectangle, draw=black, thick, fill=gray!10, text width=8.5cm, align=center, inner sep=4pt, rounded corners=3pt, font=\scriptsize}, 
    highlightbox/.style={rectangle, draw=black, very thick, fill=gray!25, text width=8.5cm, align=center, inner sep=4pt, rounded corners=3pt, font=\scriptsize},
    contextbox/.style={text width=4.5cm, font=\scriptsize\itshape, text=black, align=left, inner sep=2pt},
    arrow/.style={-{Latex[length=2mm, width=1.2mm]}, thick, draw=black!80} 
]

\node (step1) [box] {\textbf{Step 1: Problem Formulation} \\ Non-convex European market clearing model formulated as a MILP \eqref{MILP:original:one} (Properties of the dual problem: Lemma \ref{proper:dual:optimization:ALD})};

\node (step2) [highlightbox, below=0.5cm of step1] {
    \textbf{Step 2: ALD Pricing Mechanism \& Economic Properties (C1)} \\
    \vspace{4pt}
    \parbox{8cm}{\raggedright \scriptsize
    $\bullet$ \textbf{Mechanism:} Definition \ref{def:ald:pricing} (ALD Price Signals), Definition \ref{individual:optimization} (Decentralized Optimization) \\
    $\bullet$ \textbf{Incentive Compatibility:} Definition \ref{Incentive:Compatibility},  Theorem \ref{incentive:compatibility:ALD} (IC), Theorem \ref{no:pro} (PRO-Free) \\
    $\bullet$ \textbf{Revenue Adequacy:} Propositions \ref{revenue:tso} \&  \ref{TSO:revenue:adequate} (TSO of a Radial Network), Theorem \ref{individual:revenue:adequacy:supply:order} (Supply), Corollary \ref{coro:consumer} (PAO-Free Condition of Demand)
    }
};

\node (step3) [box, below=0.5cm of step2] {
    \textbf{Step 3: Computational Approach (C2)} \\
    \vspace{4pt}
    \parbox{8cm}{\raggedright \scriptsize
    $\bullet$ \textbf{Decomposition:} Definition \ref{sub:pro:definition} (MILP Decomposition) \\
    $\bullet$ \textbf{Algorithms:} Algorithm \ref{alg:primal:update} (Primal Update), Algorithm \ref{alg:com:penalty_coe} (The complete algorithm) \\
    $\bullet$ \textbf{Convergence:} Theorem \ref{convergence:modified:SAVLR} (Dual Convergence)
    }
};

\node (step4) [box, below=0.5cm of step3] {\textbf{Step 4: Comparative Analysis (C3)} \\ Comparison with benchmark pricing mechanisms \\ via an illustrative example from \citep{schiro2015convex}};

\node (step5) [box, below=0.5cm of step4] {\textbf{Step 5: Numerical Validation (C3)} \\ Nordic 12-area Electricity Market \\ Computational tractability \& Welfare equivalence};

\draw [arrow] (step1) -- (step2);
\draw [arrow] (step2) -- (step3);
\draw [arrow] (step3) -- (step4);
\draw [arrow] (step4) -- (step5);

\node (context1) [contextbox, right=0.8cm of step1] {Identifies pricing challenges (e.g., Uplift Payments) in standard marginal pricing};

\node (context2) [contextbox, right=0.8cm of step2, text=black] {\textbf{Methodological Foundation: Zero Duality Gap of MILP (Theorem \ref{exists:threshold})}};

\node (context3) [contextbox, right=0.8cm of step3] {Employs surrogate optimality condition (Definition \ref{Serro:condition}), combination technique (Remark \ref{modification:SAVLR}), and stepsize re-initialization (Remark \ref{step:re:initialization})};

\node (context4) [contextbox, right=0.8cm of step4] {Evaluates four perspectives: Incentive Compatibility, Revenue Adequacy, Economic Efficiency, and Transparency};

\node (context5) [contextbox, right=0.8cm of step5] {Demonstrates the practical viability and properties of the proposed ALD mechanism};

\draw [dashed, draw=black!50] (step1.east) -- (context1.west);
\draw [dashed, draw=black!80, thick] (step2.east) -- (context2.west); 
\draw [dashed, draw=black!50] (step3.east) -- (context3.west);
\draw [dashed, draw=black!50] (step4.east) -- (context4.west);
\draw [dashed, draw=black!50] (step5.east) -- (context5.west);

\end{tikzpicture}
\caption{Conceptual roadmap linking the paper's organization to the contributions.}
\label{fig:roadmap}
\end{figure}
\subsection{Paper organization and the conceptual roadmap}
The remainder of this paper is organized as follows. Section \ref{electricity:market:model} presents a mathematical formulation of the European day-ahead electricity market model with nonconvexities retained. In Section \ref{ALD:to:Euphemia}, we formulate this model into a compact MILP framework and apply the ALD approach. The rest of this section introduces the proposed ALD pricing mechanism and provides theoretical proofs of its core properties, including incentive compatibility, the elimination of paradoxical orders.
 The modified SAVLR algorithm, designed to compute the exact penalty coefficient and optimal Lagrangian multipliers with a convergence guarantee, is detailed in Section \ref{computational:method:zero:duality}. Section \ref{compare:ALD:and:others} provides a comparative analysis, evaluating the proposed ALD pricing mechanism against the main existing pricing rules using a standard case study. Section \ref{computational:results:ALD:stylized:example} presents the computational results of a stylized electricity market model to demonstrate the tractability of the algorithm and the economic advantages of the proposed ALD pricing mechanism. Finally, Section \ref{conclusion:future:work} concludes the paper and discusses avenues for future research. 
 \par
 To navigate the theoretical depth of this paper, Figure \ref{fig:roadmap} illustrates the conceptual roadmap of our study. To assist in understanding the ALD pricing mechanism, three illustrative examples and the proof of Lemma \ref{proper:dual:optimization:ALD} are available in the supplementary materials.
\section{MILP model of the European day-ahead electricity market}\label{electricity:market:model}
\subsection{Notations}
Mathematical Operators and Several Sets
\begin{itemize}
\item $\langle x,y \rangle$: Inner product of the vectors $x$ and $y$.
\item $\|y\|_1$: The $L_1$ norm of vector $y$, defined as $\sum_{i} |y_i|$.
\item $x \odot y$: Element-wise product (Hadamard) of vectors $x$ and $y$.
\item $\mathcal{A}, \mathcal{T}, \mathcal{L}, \mathcal{LS}$: Sets of bidding areas, time slots ($|\mathcal{T}|=24$), interconnections, and small line sets, respectively.
\end{itemize}
Order-Related Index Sets
\begin{itemize}
\item $\mathcal{O}_{E}, \mathcal{O}_{BO}, \mathcal{O}_{PB}, \mathcal{O}_{RB}, \mathcal{O}_{FHB}$: The sets of elementary orders, block orders, profile block orders, regular block orders, and flexible hourly orders, respectively, where $\mathcal{O}_{BO} = \mathcal{O}_{PB} \cup \mathcal{O}_{RB}$.
\item Subscripts $s$ and $d$ denote the supply and demand subsets of any set $\mathcal{O}$ (e.g., $\mathcal{O}_{E,s}, \mathcal{O}_{E,d}$, where $\mathcal{O}_{E}=\mathcal{O}_{E,s}\cup \mathcal{O}_{E,d}$).
\item $\mathcal{O}_{slg}$: The set of supply orders subject to load gradient conditions, where $\mathcal{O}_{slg} \subset \mathcal{O}_{E,s}$.
\item Subscripts $a$ and superscript $t$ denote subsets of any set $\mathcal{O}$ subject to the area $a$, or the time slot $t$, (e.g., $\mathcal{O}_{E,s,a}^t, \mathcal{O}_{E,d,a}^t$, where $\mathcal{O}_{E,s}=\bigcup_{a\in \mathcal{A}}\bigcup_{t\in \mathcal{T}}{\mathcal{O}_{E,s,a}^t}, \mathcal{O}_{E,d}=\bigcup_{a\in \mathcal{A}}\bigcup_{t\in \mathcal{T}}{\mathcal{O}_{E,d,a}^t}$).
\end{itemize}
Parameters
\begin{itemize}
\item $p_{a}^t(i), q_{a}^t(i)$: Price and quantity of order $i$ at area $a$ during period $t$.
\item $G_s^{dn}, G_s^{up}$: Decrease and increase gradient for supply order $s$ ($MW/min$).
\item $R^{H}, R^{D}$: Hourly ($H$) and daily ($D$) ramp limits for net injections, interconnections, or line sets.
\item $P^{ini}, F^{ini}$; $F_{lt}^{min}, F_{lt}^{max}$: Initial net injection and initial power flow at the start of the trading day; Minimum and maximum available transmission capacity for line $l$ at period $t$.
\item $H_{aa'}^{lt}$; $A_{LB}, A_{EG}$: Power Transfer Distribution Factor (PTDF) matrix for AC interconnections; Incidence matrices for linked block orders and exclusive groups.\end{itemize}
Decision Variables
\begin{itemize}
\item $x_{E,a}^{t}(i) \in [0, 1]$: Acceptance ratio for a continuous elementary order $i$ at area $a$ and time $t$.
\item $x_{RB,a}(i), x_{FHB,a}^t(i) \in \{0, 1\}$: Binary acceptance statuses for a regular and flexible hourly block order $i$ at area $a$ and time $t$.  
\item $u_{PB,a}(i) \in \{0, 1\}$: Binary execution status for a profile block order $i$.
\item $p_{at}, f_{lt}$: Net injection of area $a$ and power flow on interconnection $l$ during time period $(t-1,t)$.
\item $e_{aa'}^t$: Bilateral energy exchange between area $a$ and $a'$.
\item $x^t_a(i)$: A general order $i$ at area $a$ and time $t$.  (Introduced for ease of notation, either an element of $x_{RB,a}, x_{PB,a}, x_{FHB,a}^t$ or $x_{E,a}^t$).
\end{itemize}
\subsection{The electricity market model and the discrete optimization problem}
Note that while we adopt the standard acronym MILP (Mixed-Integer Linear Programming) for convention, our formulation is strictly a Mixed-Binary Linear Programming (MBLP) problem.
\par
{\em Unified Formulations:}
For each area $a \in \mathcal{A}$, we aggregate block order acceptance ratios into a single vector $x_{BO,a} = [x_{PB,d,a}, x_{RB,d,a}, x_{PB,s,a}, x_{RB,s,a}]^\top$, where $x_i \in \{0\} \cup [R_i^{min}, 1]$ for profile blocks ($i \in \mathcal{O}_{PB,a}$) and $x_i \in \{0, 1\}$ for regular blocks ($i \in \mathcal{O}_{RB,a}$). Correspondingly, the economic valuation vector for order type $j \in \{E, BO, FHB\}$ is defined as $c_{j,a}^t = p_{j,a}^t \odot q_{j,a}^t$ (in EUR). Here, $c_i^t \le 0$ for demand orders and $c_i^t \ge 0$ for supply orders. Note that for flexible hourly orders ($j = FHB$), $c_{FHB,a}$ is time-invariant.
\par
Social welfare maximization is formulated as a cost-minimization problem, with the model benchmark from \citep{chatzigiannis2016european}. To ensure the model reflects practical market operations, we incorporate the following configurations:
(i) Storage Assets: Utility-scale storage is decoupled into separate demand (charging) and supply (discharging) orders across different periods $(t_1, t_2)$.
 (ii) Cost Structure: Start-up costs are internalized within the supply offer prices of supply block orders.
 (iii) Operational Assumptions: All interconnections are treated as AC lines operated by non-profit entities.
In addition, the minimum income condition and the scheduled stop condition are not included in the model. 
The cost minimization problem is given by
\begin{subequations}\label{MILP:original:one}
\allowdisplaybreaks
\begin{align}
\min \quad & \sum_{a\in\mathcal{A}} \sum_{t\in\mathcal{T}} \Bigl( \langle c_{E,a}^t, x_{E,a}^t \rangle + \langle c_{BO,a}^t, x_{BO,a} \rangle + \langle c_{FHB,a}, x_{FHB,a}^t \rangle \Bigr), \label{obj:1} \\
\text{s.t.} \quad & p_a^t = \langle q_{E,a}^t, x_{E,a}^t \rangle + \langle q_{BO,a}^t, x_{BO,a} \rangle + \langle q_{FHB,a}, x_{FHB,a}^t \rangle, \quad \forall\, t\in\mathcal{T}, a\in\mathcal{T} ,\label{network:injection:equation} \\
& \sum_{a \in \mathcal{A}} p_a^t = 0, \quad p_a^t = \sum_{a' \in \mathcal{A}} (e_{aa'}^t - e_{a'a}^t), \quad \forall\, t\in\mathcal{T}, a\in\mathcal{A} ,\label{power:balance_flow} \\
& f_{lt} = \sum_{a, a'} H_{aa'}^{lt} e_{aa'}^t, \quad F_{lt}^{\min} \leq f_{lt} \leq F_{lt}^{\max}, \quad \forall\, l\in\mathcal{L}, t\in\mathcal{T}, \label{power:flow:limits} \\
& R_{PB,a}^{\min}(i)\, u_{PB,a}(i) \leq x_{PB,a}(i) \leq u_{PB,a}(i), \quad \forall\, i \in \mathcal{O}_{PB,a},\label{block:order:constr:1} \\
& x_{BO,a} \leq A_{LB} x_{BO,a}, \quad A_{EG} x_{BO,a} \leq 1, \quad \forall\, a\in\mathcal{A},\label{block:order:constr:2} \\
& \sum_{t \in \mathcal{T}} x_{FHB,a}^t(i) \leq 1, \quad \forall\, i \in \mathcal{O}_{FHB,a}, \label{block:order:constr:3} \\
& | q_{E,s,a}^{t}(i) x_{E,s,a}^{t}(i) - q_{E,s,a}^{t-1}(i) x_{E,s,a}^{t-1}(i) | \leq 60 G_s^{up/dn}, \forall\, i\in \mathcal{O}_{slg,a}, t\in \mathcal{T}\setminus\{1\} \label{load:gradient} \\
& | p_{at} - p_{a,t-1} - p_{at}^{H,ini} | \leq R_{at}^{H,up/dn}, \quad \forall \, a\in\mathcal{A}, t\in\mathcal{T}, \label{ramping:p:hour} \\
& | \sum_{t} p_{at} - p_a^{D,ini} | \leq R_a^{D, up/dn}, \quad \forall\, a\in\mathcal{A}, \label{ramping:p:daily}\\
& | f_{lt} - f_{l,t-1} - F_{lt}^{H,ini} | \leq R_{lt}^{H,up/dn}, \quad \forall\, l\in\mathcal{L}, t\in\mathcal{T}, \label{ramping:f:single} \\
& | \sum\limits_{l\in L_{ls}}{f_{lt}} - \sum\limits_{l\in L_{ls}}{f_{l,t-1}} - F_{lt}^{H,ini} | \leq R_{lt}^{H,up/dn}, \quad \forall\, ls\in LS, t\in\mathcal{T}. \label{ramping:f:Line:set} 
\end{align}
\end{subequations}
  \paragraph{Interpretations of the optimization problem \eqref{MILP:original:one}}
  \begin{itemize}
  \item Objective Function \eqref{obj:1}: Minimizes the total system costs (or maximizes social welfare), aggregating bids from Elementary orders $x_{E,a}^t$, Block Orders (BO) $x_{BO,a}$, and Flexible Hourly Orders (FHB) $x_{FHB,a}^t$.
  \item Power Balance and Network Physics \eqref{network:injection:equation}--\eqref{power:flow:limits}: Equation \eqref{network:injection:equation} calculates the net nodal injection $p_a^t$ by summing all cleared volumes. Market-wide balance and nodal flow conservation are given by \eqref{power:balance_flow}. Transmission constraints \eqref{power:flow:limits} enforce physical capacity limits ($F^{\min/\max}$) based on the PTDF matrix $H$.
  \item Constraints of Orders \eqref{block:order:constr:1}--\eqref{block:order:constr:3}: Capturing the complex logic of European orders. 
  \begin{itemize}
        \item Profile and Regular Blocks: \eqref{block:order:constr:1} restricts profile block acceptance ratios to $x \in \{0\} \cup [R^{min}, 1]$. Regular blocks are treated as a special case where $R^{min} = 1$, resulting in a binary outcome.
        \item \textit{Dependencies and Exclusivity}: \eqref{block:order:constr:2} enforces parent-child requirements for linked orders, while \eqref{block:order:constr:3} ensures mutual exclusivity within predefined groups.
        \item \textit{Flexible Orders}: \eqref{block:order:constr:3} confirms that flexible hourly orders are scheduled in at most one time period across the trading day.
    \end{itemize}
        \item Operational Gradients: Constraints \eqref{load:gradient} impose load gradient limits on supply elementary orders.
    \item System Ramping: The set of inequalities \eqref{ramping:p:hour}--\eqref{ramping:f:Line:set}  characterizes multi-temporal ramping limits, including hourly and daily constraints on net injections (\eqref{ramping:p:hour}, \eqref{ramping:p:daily}), individual interconnection flow \eqref{ramping:f:single}, and flows for specified line sets \eqref{ramping:f:Line:set}.
  \end{itemize}
Due to the block orders in \eqref{block:order:constr:1}--\eqref{block:order:constr:3}, the standard marginal pricing often requires discriminatory uplift payments.
\section{The ALD-based pricing model}\label{ALD:to:Euphemia}
To simplify the application of the ALD approach, we first formulate the constraints of \eqref{MILP:original:one} into compact feasible sets.
\paragraph{Decomposed Variable and the Feasible Set}We decompose the decision variables into block vectors $x_E$, $x_{BO}$, $x_{FHB}$, $p$ representing elementary, block, flexible hourly orders, and the net injections across all areas and periods. The feasible sets of each decomposed variable are formulated as follows:
\begin{itemize}
\item $X_E$: Defined by $x_{E,a}^{t}(i) \in [0, 1]$ and load gradient constraints \eqref{load:gradient}.
\item $X_{BO}$: $x_{PB,a}(i)$ with disjunctive set $\{0\} \cup [R_{BO}^{min}, 1]$ (e.g, \eqref{block:order:constr:1} and  $x_{RB,a}(i) \in \{0,1\}$, subject to logic constraints \eqref{block:order:constr:2}.
\item $X_{FHB}$: $x^t_{FHB,a}(i) \in \{0, 1\}$ with temporal constraint \eqref{block:order:constr:3}.
\item $P$: $p_a^t$ defined by the power balance \eqref{power:balance_flow}, flow limits \eqref{power:flow:limits}, and all system ramp limits \eqref{ramping:p:hour}--\eqref{ramping:f:Line:set}, which is detailed in Section \ref{sec:deco:sub:surro:sub}.
\end{itemize}
The feasible set of the problem \eqref{MILP:original:one} is defined by $X \coloneqq X_E \times X_{BO} \times X_{FHB}$, where $X_E$ and $P$ are polyhedra; $X_{BO}$ and $X_{FHB}$ contain integer constraints, characterizing the problem as an MILP.
\paragraph{Compact formulation}Using aggregated notation, the optimization problem \eqref{MILP:original:one} is formulated as
\begin{subequations}\label{mix:integer:linear:program}
\begin{align}
\min_{x \in X, p \in P} \sum_{a \in \mathcal{A}} \sum_{t \in \mathcal{T}} & \left( \langle c_{E,a}^t, x_{E,a}^t \rangle + \langle c_{BO,a}^t, x_{BO,a} \rangle + \langle c_{FHB,a}, x_{FHB,a}^t \rangle \right) \label{obj:shorten} \\
\text{s.t. } \quad p_a^t = & \langle q_{E,a}^t, x_{E,a}^t \rangle + \langle q_{BO,a}^t, x_{BO,a} \rangle + \langle q_{FHB,a}, x_{FHB,a}^t \rangle, \quad \forall~ a, t. \label{eq:global_balance}
\end{align}
\end{subequations}
where \eqref{eq:global_balance} represents the global coupling constraints that link the acceptance decisions of all orders with the net injection set $P$.
\paragraph{ALD Formulation}We now apply the ALD approach by dualizing the coupling constraints \eqref{eq:global_balance} with multipliers $\lambda_a^t$ and augmenting the objective with an $L_1$-norm penalty term with coefficient $\rho$. The ALD problem is defined as
\begin{equation}\label{mix:integer:linear:program:ALD}
\max_{\lambda \in \mathbb{R}^{|\mathcal{A}| \times |\mathcal{T}|}} z_\rho^{ALD}(\lambda) \coloneqq \max_{\lambda} \min_{x \in X, p \in P} L_\rho(x, p, \lambda),
\end{equation}
where the Augmented Lagrangian function $L_\rho$ is
\begin{equation}\label{central:obj}
\begin{aligned}
L_\rho(x, p, \lambda) := & \sum_{a, t} \left( \sum_{j \in {E, BO, FHB}} \langle c_{j,a}^t, x_{j,a}^t \rangle \right) + \sum_{a, t} \lambda_a^t \left( p_a^t - \sum_{j} \langle q_{j,a}^t, x_{j,a}^t \rangle \right) \\
& + \rho \sum_{a, t} \left\| p_a^t - \sum_{j} \langle q_{j,a}^t, x_{j,a}^t \rangle \right\|_1.
\end{aligned}
\end{equation}
\subsection{Strong duality of the ALD approach}
Strong duality is an important property in pricing mechanisms, as argued in Section \ref{sec:intro}. This section begins with a lemma that introduces the properties of the dual optimization problem for \eqref{mix:integer:linear:program:ALD}.
\begin{lemma}[Characterization of the MILP Dual]\label{proper:dual:optimization:ALD}
    The dual optimization problem $\max\limits_{\lambda \in \mathbb{R}^{|\mathcal{A}|\times |\mathcal{T}|}} {z_\rho^{ALD}(\lambda)}$ of \eqref{mix:integer:linear:program:ALD} is concave, where the objective function is piecewise linear and sub-differentiable.
\end{lemma}

Note that Lemma 1 is a classical result, and the proof has been omitted. The properties of the dual problem of a Lagrange relaxation of the MILP are stated in \citep[p.117]{bragin2015convergence}, and the additional penalty term does not influence these properties, see the detailed proof in the Supplement. 
\par
In the following, we will introduce a strong duality theory for the ALD approach described above for the MILP problem. In addition, readers interested in closing the duality gap can refer to \citep{chen2010extended,lefebvre2024exact}.  
Note that, unlike the augmented Lagrangian relaxation procedure, the approach of \citep{chen2010extended} requires only a penalty.
\begin{theorem}[Exact Penalty Coefficient]\label{exists:threshold}Let $(x^*, p^*)$ be an optimal solution to the MILP problem \eqref{mix:integer:linear:program} with optimal objective value $z^{MIP}$. Define $z_\rho^{ALD}(\lambda)$ as the optimal value of the dual problem \eqref{mix:integer:linear:program:ALD}.\begin{enumerate}\item If $\lambda^*$ maximizes the dual function, there exists a finite threshold $0 < \rho^* < \infty$ such that $z_{\rho}^{ALD}(\lambda^*) = z^{MIP}\ \forall \rho \geq \rho^*$, effectively closing the duality gap \citep[Theorem 4]{feizollahi2017exact}.\item For any arbitrary $\lambda \in \mathbb{R}^{|\mathcal{A}|\times |\mathcal{T}|}$, there exists a finite $\rho(\lambda) > 0$ such that $z_{\rho(\lambda)}^{ALD}(\lambda) = z^{MIP}$ \citep[Proposition 8]{feizollahi2017exact}.
\end{enumerate}
\end{theorem}
Theorem \ref{exists:threshold} establishes that for a given $\lambda^*$, a threshold $\rho^*$ exists that eliminates the duality gap. Furthermore, for any $\rho > \rho^*$, the optimal solution set of the ALD-based problem \eqref{mix:integer:linear:program:ALD} coincides with that of the original MILP \eqref{mix:integer:linear:program}. In this context, $\lambda^*$ is said to support an exact penalty representation for the MILP problem \citep[Proposition 1]{feizollahi2017exact}.
\par Here, a fundamental question arises, for $\lambda^*\in \mathbb{R}^{|\mathcal{A}|\times |\mathcal{T}|}$, one can set a very large penalty coefficient to close the duality gap. However, such a penalty coefficient will cause computational issues, and more importantly, it is meaningless when interpreting in practice, see the discussion in \citep{lefebvre2024exact}. Hence, our aim is to determine a valid penalty coefficient $\rho \in [\rho^*, \overline{\rho}]$ that is larger than the threshold $\rho^*$, but remains as close to this lower bound as computationally feasible.
\subsection{The proposed ALD pricing mechanism}\label{ALD:pricing:proposing}
Assume that the exact penalty coefficient $\rho$ for the ALD problem \eqref{mix:integer:linear:program:ALD} has been determined such that $\rho \in [\rho^*, \overline{\rho}]$, ensuring strong duality. Let $\Lambda$ denote the set of optimal dual variables $\lambda^{t,*} \in \mathbb{R}^{|\mathcal{A}| \times |\mathcal{T}|}$ for problem \eqref{MILP:original:one}. Since problem \eqref{mix:integer:linear:program} is an MILP, its optimal solution $(x^*, p^*)$, which does not need to be unique, can be efficiently obtained using commercial solvers such as Gurobi or CPLEX. Based on these parameters, we define the ALD pricing mechanism as follows:
\begin{definition}[ALD Price Signals]\label{def:ald:pricing}
At each time slot $t \in \mathcal{T}$, the locational marginal price (LMP) of the bidding area $a \in \mathcal{A}$, denoted by $LMP_a^t$, is defined as
\begin{equation}\label{LMP:ALD:Pricing}
\begin{aligned}
LMP_a^t&\coloneqq \eta_a^t + \min_{\lambda_a^{t,*}\in \Lambda}{\lambda_i^{t,*}},\\
s.t.& ~ \lambda_a^{t,*} +\eta_a^t\in \Lambda,\eta_a^t\leq \rho,~ a\in \mathcal{A},~t\in \mathcal{T},
\end{aligned}
\end{equation}
where $\eta_a^t$ is a parameter used to derive congestion price signals, while excessive transmission incentives are mitigated by the upper bound $\rho$.  
\par
For a supply(demand) order $i$, its revenue(cost) consists of the commodity revenue(cost): $LMP_a^t q^t(i)$\\$ x^{t}(i)$, a non-convex reward(charge): $\rho q^t(i) x^{t,*}(i)$, and a deviation penalty: $\rho |q^t(i) (x^t(i) - x^{t,*}(i))|$.
\par
The Transmission System Operator (TSO) collects congestion revenue based on the Financial Transmission Right (FTR), defined as $FTR_{aa'}^t \coloneqq  LMP_{a'}^t- LMP_a^t $, where power flows through line $l$ from node $a$ to node $a'$.
\end{definition}
According to Def. \ref{def:ald:pricing}, the aggregate daily surplus for an  order $i$ is
\begin{equation}\label{surplus:formula}
\begin{aligned}
S_i = \sum_{t \in \mathcal{T}} &\left[ LMP_a^t \, q^t(i) x^{t}(i) - \rho |q^t(i) \left(x^t(i) - x^{t,*}(i)\right)|\right. \\
&\left.+~ \rho q^t(i) x^{t,*}(i)-  p^t(i) q^t(i) x^{t}(i) \right].
\end{aligned}
\end{equation}
The overall daily welfare for the TSO, assuming no flow deviation, is $\sum_{l \in \mathcal{L}}\sum_{t \in \mathcal{T}} f_{lt}^* ~FTR_l^t$, where $f_{lt}^*$ is the flow of line $l$ during the time period $[t-1,t]$.
\par
Note that in Eq.~\eqref{LMP:ALD:Pricing}, the minimal optimal dual variables are selected to ensure the economic efficiency of the pricing mechanism.
Furthermore, the introduction of $\eta_a^t$ facilitates congestion price discovery when network constraints are active, and its selection is problem-dependent.
\subsubsection{Decentralized optimization, incentive compatibility and congestion rents}\label{individual:rationality:property}
The problem in \eqref{mix:integer:linear:program} represents the centralized market-clearing formulation, from which we derive the individual optimization problems. For notational consistency, $\lambda_a^{t,*}$ denotes the locational marginal price (LMP) as determined by the expression \eqref{LMP:ALD:Pricing}.
\begin{definition}[Decentralized Optimization]\label{individual:optimization}
The individual optimization problem for each order $i$ in bidding area $a \in \mathcal{A}$ is formulated as follows:
\begin{equation}\label{individual:obj}
\begin{aligned}
    &\min_{x_a(i) \in X_i} \sum_{t\in\mathcal{T}} \left\{ c_a^t(i)x_a^t(i) - \lambda_a^{t,*} q_a^t(i)x_a^t(i) + \rho \left| q_a^t(i) \Delta x_a^t(i) \right| - \rho q_a^t(i)x_a^{t,*}(i) \right\},\\
    &\text{or equivalently } \\
    &\max_{x_a(i) \in X_i} \sum_{t\in\mathcal{T}} \left\{ \lambda_a^{t,*} q_a^t(i)x_a^t(i) - c_a^t(i)x_a^t(i) - \rho \left| q_a^t(i) \Delta x_a^t(i) \right| + \rho q_a^t(i)x_a^{t,*}(i) \right\}, \\
    &\text{s.t. }\qquad \Delta x_a^t(i) = x_a^t(i) - x_a^{t,*}(i), \quad \forall~ t \in \mathcal{T}, a \in \mathcal{A}, \\
    &\qquad\quad~~ \lambda_a^{t,*} \text{by Def. }  \ref{def:ald:pricing}, \quad \rho \in [\rho^*, \overline{\rho}],
    \end{aligned}
\end{equation}
where $X_i$ is the feasible region of $x_a(i) = [x_a^1(i),\cdots, x_a^{|\mathcal{T}|}(i)]$, and for block orders: $x_a^1(i)=\cdots= x_a^{|\mathcal{T}|}(i)$.
\end{definition}
 Note that in Eq. \ref{individual:obj}, the sets $X_i$ may exhibit interdependencies, as seen in the block order constraints \eqref{block:order:constr:2}--\eqref{block:order:constr:3}. For linked block orders, a hierarchical solving sequence is required: the optimization problems for parent orders must be solved before those for child orders. Furthermore, orders subject to exclusive group constraints are coupled and must be solved simultaneously within a single optimization framework to ensure feasibility.

Building on the definition of Incentive Compatibility in the literature, the theorem below establishes a key property inherent in the proposed ALD pricing mechanism.
\begin{definition}[Incentive Compatibility]\label{Incentive:Compatibility}
    A pricing mechanism is incentive-compatible if no market participant has a financial incentive to deviate from its allocated quantity (the centralized solution), \citep[Section 4.3]{milgrom2022linear}.
\end{definition}
\begin{theorem}[Incentive Compatibility of ALD]\label{incentive:compatibility:ALD}
The ALD pricing mechanism introduced in Def. \ref{def:ald:pricing} is incentive-compatible (IC). 
\end{theorem}
\begin{proof}
Consider the minimization problem in \eqref{individual:obj}, then the objective function represents the additive inverse of the order $i$'s surplus. We establish incentive compatibility by proving that the market-clearing solution $\{x_a^{t,*}(i)\}_{t\in \mathcal{T}}$ minimizes the individual optimization \eqref{individual:obj}. Let $\sum_{a, t, i}$ denote $\sum_{a,t}\sum_{i \in \mathcal{O}_a^t}$.

\paragraph{Step 1: The Strong Duality of the Centralized ALD} 
Under the ALD pricing mechanism with $\lambda^*$ and $\rho\in [\rho^*, \overline{\rho}]$, the dual problem \eqref{mix:integer:linear:program:ALD} achieves its maximum at $z^{MIP}$, [see Theorem \ref{exists:threshold}]. For any optimal decision variable $(x^*, p^*)$, the centralized objective \eqref{central:obj} satisfies
\begin{equation}\label{central:L:def}
    L_{\rho}(x, p, \lambda^*) = \sum_{a, t, i} c_a^t (i)x_a^t(i) + \sum_{a, t} \left( \lambda_a^{t,*} r_a^t(x,p) + \rho |r_a^t(x,p)| \right) = z^{MIP}.
\end{equation}
where $r_a^t(x,p) = p_a^t - \sum_{i \in \mathcal{O}_a^t} q_a^t(i)x_a^t(i)$ is the nodal power balance residual.

\paragraph{Step 2: Aggregation of the Individual Optimization Problem}
\par 
 Reformulate the objective function denoted as $f_{a,i}(x_a(i))$ in \eqref{individual:obj} by adding $\sum_{t}{\lambda_a^{t,*}q_a^t(i)x_a^{t,*} (i)} - \sum_{t}{\lambda_a^{t,*}q_a^t(i)x_a^{t,*} (i)}$ (=0):
\begin{equation}\label{individual:obj:reformulate}
    \begin{aligned}
        &\sum_{t}{c_a^t (i)x_a^t(i)} +  \sum_{t}{\lambda_a^{t,*}q_a^t(i)x_a^{t,*} (i)} - \sum_{t}{\lambda_a^{t,*} q_a^t(i)x_a^t (i)}\\
    &+\rho\sum_{t}{\left|q_a^t(i) x_a^t(i)-q_a^t(i)x_a^{t,*}(i)\right|} 
    - \rho\sum_{t}{q_a^t(i)x_a^{t,*}(i)}
    -  \sum_{t}{\lambda_a^{t,*}q_a^t(i)x_a^{t,*} (i)}. \\
    \end{aligned}
\end{equation}
By aggregating the objective functions in \eqref{individual:obj:reformulate} across all orders, we observe that the summation of the second to last term satisfies
\begin{equation}\label{constant:term}
\begin{aligned}
\rho\sum\limits_{a, t, i}{q_a^t(i)x_a^{t,*}(i)} = \sum\limits_{t}\rho\sum\limits_{a}\sum\limits_{i}{q_a^t(i)x_a^{t,*}(i)}=0,
\end{aligned}
\end{equation}
which is due to the power balance equation \eqref{power:balance_flow} $\forall ~t\in \mathcal{T}$. 
The summation of the final term in \eqref{individual:obj:reformulate}, specifically $-\sum_{a, t, i }{\lambda_a^{t,*}q_a^t(i)x_a^{t,*}(i)}$, represents the total congestion revenue allocated to the transmission system operator (TSO). This relationship is formally established and demonstrated in Proposition \ref{revenue:tso}.
Therefore, the summation of Eq. \eqref{individual:obj:reformulate} over all orders minus the revenue of the transmission system operator will be equal to
\begin{equation}\label{individual:obj:reformulate:sum}
\allowdisplaybreaks
    \begin{aligned}
         \Phi_{aggre}(x) =\sum\limits_{a, t, i}{c_a^t (i)x_a^t(i)} &+  \sum\limits_{a, t, i}{\lambda_a^{t,*}\left(q_a^t(i)x_a^{t,*} (i) - q_a^t(i)x_a^{t} (i)\right)}
    +\rho\sum\limits_{a, t, i}{\left|q_a^t(i)x_a^t(i)-q_a^t(i)x_a^{t,*}(i)\right|}.  \\
    \end{aligned}
\end{equation}
\paragraph{Step 3: The Triangle Inequality of the Centralized ALD}
Note that given the optimal decision variables, the residuals $r_a^{t,*}$ vanish, which implies $p_a^{t,*} = \sum_i q_a^t(i)x_a^{t,*} (i)$. Incorporating this zero constant into \eqref{central:L:def},
\begin{equation}\label{central:obj:add}
\begin{aligned}
   L_{\rho}(x, p, \lambda^*)= &\sum\limits_{a, t, i}{c_a^t (i)x_a^t(i)} +  \sum\limits_{a, t}{\lambda_a^{t,*}\left(r_a^{t}- r_a^{t,*}\right)}
    +\rho \sum\limits_{a, t}{\left|r_a^{t,*}- r_a^{t,*}\right|}. 
\end{aligned}
\end{equation}
Applying the triangle inequality to the RHS of \eqref{central:obj:add}:
\begin{equation}\label{central:obj:leq}
\begin{aligned}
  L_{\rho}(x, p, \lambda^*) \leq &\sum_{a, t, i} c_a^t (i)x_a^t(i) + \sum_{a, t} \lambda_a^{t,*}(p_a^t - p_a^{t,*})  
    + \sum_{a, t} \lambda_a^{t,*} \sum_{i \in \mathcal{O}_a^t} \left(q_a^t(i)x_a^{t,*} (i) - q_a^t(i)x_a^t(i)\right)  \\
   & + \rho \sum_{a, t} \left( |p_a^t - p_a^{t,*}| + \sum_{i \in \mathcal{O}_a^t} |q_a^t(i)x_a^t(i) - q_a^t(i)x_a^{t,*}(i)| \right). 
\end{aligned}
\end{equation}
The RHS of Eq. \eqref{central:obj:leq} has two additional terms as follows, compared to Eq. \eqref{individual:obj:reformulate:sum}:
\begin{equation}\label{something:left}
    \begin{aligned}
        \sum\limits_{a, t}{\lambda_a^{t,*}(p_a^t- p_a^{t,*})} + \rho \sum\limits_{a, t}{|p_a^t- p_a^{t,*}|}.
    \end{aligned}
\end{equation}
Recall the definition of the linear set $P$ in Section \ref{ALD:to:Euphemia}, $\forall~p_a^t\in P$, it satisfies $\sum_{a,t}{p_a^t} = 0$. By adding the term $\sum_{a, t}{p_a^t}$ to Eq. \eqref{something:left}, and considering the constraint $p_a^t\in P$, the net injection subproblem is formulated as
\begin{equation}\label{net:injection:ALD}
    \begin{aligned}
     \min\limits_{p\in P} \Psi(p) &:= \min\limits_{p\in P} \left( \sum\limits_{a, t}{p_a^t}+
        \sum\limits_{a, t}{\lambda_a^{t,*}(p_a^t- p_a^{t,*})} + \rho \sum\limits_{a, t}{|p_a^t- p_a^{t,*}|}\right).
    \end{aligned}
\end{equation}
 Replace above in Eq. \eqref{central:obj:leq}, we have the following bound:
\begin{equation}\label{inequality:bridge}
     L_{\rho}(x, p,\lambda^*) \leq \Phi_{aggre}(x) + \Psi(p) - \sum_{a,t}{p_a^t}.
\end{equation}
For the optimal dispatch $p^*\in P$, $\Psi(p^*) = 0$, and therefore, the optimal objective function value of \eqref{net:injection:ALD} is always non-positive, and thus\\
\begin{equation}\label{inequality:bridge:final}
\begin{aligned}
    z^{MIP} =  \min_{x\in X, p\in P}{\mathcal{L}_{\rho}(x, p,\lambda^*)} &\leq \min_{x\in X}\Phi_{aggre}(x) + \min_{p\in P}\Psi(p)- \min_{p\in P}{\sum_{a,t}{p_a^t}}\\
    &\leq \min_{x\in X}{\Phi_{aggre}(x)}.
    \end{aligned}
\end{equation}
\paragraph{Step 4: Contradiction and Optimality}
Assume an order $i$ can improve its welfare by deviating to $x_a(i) \neq x_a^*(i)$. This implies $f_{a,i}(x_a(i)) < f_{a,i}(x_a^*(i))$, and consequently $\min_{x\in X}\Phi_{aggre}(x) < \Phi_{aggre}(x^*) = z^{MIP}$. However, this contradicts the lower bound $z^{MIP}$ established in \eqref{inequality:bridge:final}. Therefore, $\{x_a^{t,*}(i)\}_{t\in \mathcal{T}}$ is the optimal solution for each decentralized problem, thus the ALD mechanism is incentive-compatible.
\end{proof}
Following the proof of Theorem \ref{incentive:compatibility:ALD}, the following proposition characterizes the revenue for transmission system operators.
\begin{proposition}[Aggregate  Congestion Rent]\label{revenue:tso}
The summation of the final term in \eqref{individual:obj:reformulate} over all orders, given by $-\sum_{a, t, i} \lambda_a^{t,*} q_a^t(i) x_a^{t,*}(i)$, equals the total congestion revenue allocated to the Transmission System Operator. 
\end{proposition}
\begin{proof}
First, recall the optimal net injection $p_a^{t,*} = \sum_{i \in \mathcal{O}_a^t} q_a^t(i) x_a^{t,*}(i)$ (e.g., \ref{network:injection:equation}). The total revenue term can be expressed in a vector form as
\begin{equation}\label{total:tso:revenue_revised}
-\sum_{a,t,i} \lambda_a^{t,*} q_a^t(i) x_a^{t,*}(i) = -\sum_{t \in \mathcal{T}} \sum_{a \in \mathcal{A}} \lambda_a^{t,*} p_a^{t,*} = -\sum_{t \in \mathcal{T}} \lambda^{t,* \top} p^{t,*}.
\end{equation}
Following the nodal power balance $p_t^* = B^\top f_t^*$ \citep{wood2013power}, where $B \in \mathbb{R}^{|L| \times |\mathcal{A}|}$ is the line-node incidence matrix and $f_t^*$ is the optimal flow vector, we obtain
\begin{equation}
-\sum_{t\in \mathcal{T}} \lambda^{t,*\top} (B^\top f_t^*) = \sum_{t\in \mathcal{T}} (B\lambda^{t,*})^\top (-f_t^*).
\end{equation}
The term $B\lambda^{t,*}$ yields a vector of dual price differentials across each interconnection (e.g., $\lambda_{a'}^{t,*} - \lambda_{a''}^{t,*}$ for a line from $a'$ to $a''$). Therefore, $(B\lambda^{t,*})^\top (-f_t^*)$ is the aggregate congestion rent during $[t-1, t]$. This concludes the proof.
\end{proof}
From Proposition \ref{revenue:tso}, we can deduce that the individual optimization objectives defined in Def.~\ref{individual:optimization} effectively internalize transmission rents through the ALD congestion prices established in Def.~\ref{def:ald:pricing}.
\begin{proposition}[Revenue Adequacy for TSO of Radial Networks]\label{TSO:revenue:adequate}
The aggregate daily congestion rent of a radial network under the ALD pricing mechanism is positive \footnote{Note that the revenue adequacy property may not hold for the TSO of a general network structure.}.
\end{proposition}
\begin{proof}
The proof of Proposition \ref{TSO:revenue:adequate} follows from that the energy always flows from a low-price area to a high-price area (i.e., $\lambda_{a'}^{t,*} < \lambda_{a''}^{t,*}$) for a radial network structure, which is the same as the convex locational marginal pricing.
\end{proof}
\subsubsection{Individual revenue adequacy}\label{individual:revenue:adequacy}
In this section, we examine whether the ALD pricing mechanism inherently guarantees individual revenue adequacy—a property often lacking in existing pricing models such as the IP pricing mechanism. While some mechanisms, such as the completely positive programming (CPP) based approach \citep{guo2025copositive} and the primal-dual pricing framework \citep{rui2012}, enforce revenue adequacy via explicit constraints (i.e., surplus $\geq 0$), such a ``forced'' formulation can inadvertently compromise incentive compatibility \citep{guo2025copositive}. In contrast, the ALD mechanism ensures non-negative surplus for supply orders as an intrinsic structural outcome. Although this guarantee does not universally extend to the demand orders (see Corollary~\ref{coro:consumer}). The following theorem establishes this property for all supply orders.
\begin{theorem}[Non-negativity of Supply-Order Surplus]\label{individual:revenue:adequacy:supply:order}
Assuming zero deviation for a supply order $i$ (i.e., $x_a^t(i) = x_a^{t,*}(i),\forall ~t\in \mathcal{T}$), its daily surplus is non-negative under the ALD pricing mechanism defined in Def. \ref{def:ald:pricing}.
\end{theorem}
\begin{proof}
Let $\pi_a^t = LMP_a^t + \rho$ denote the effective settlement price under the ALD pricing mechanism. For a supply order $i$ in bidding area $a$, the daily surplus is defined as
\begin{equation}\label{supply:surplus:def}S_{a,i} = \sum_{t\in\mathcal{T}} \left( \pi_a^t q_a^t(i) - c_a^t(i) \right) x_a^{t}(i),
\end{equation}
where $c_a^t(i) = p_a^t(i) q_a^t(i)$. 
Recall the individual optimization problem with objective function $f_{a,i}$ in \eqref{individual:obj}. For the supply order $i$, we have $q_a^t(i) > 0$. By comparing the optimal dispatch $x_a^{*}(i)$ with the null dispatch $0 \in X_i$, we observe:
\begin{enumerate}
\item $f_{a,i}(0) = 0$: In that case, the nonconvex reward of the supply order $i$ is precisely settled due to the deviation. 
\item $f_{a,i}(x_a^*(i)) = -S_{a,i}$: When evaluated at the optimal solution $x_a^{*}(i)$ (with zero deviation), the minimization objective in \eqref{individual:obj} reduces to the additive inverse of the surplus.
\end{enumerate}
Since $x_a^{*}(i)$ is the optimal solution of \eqref{individual:obj}, (see the proof of Theorem \ref{incentive:compatibility:ALD}), it follows that $f_{a,i}(x_a^*(i)) \leq f_{a,i}(0)$. Consequently, $-S_{a,i} \leq 0$, which implies $S_{a,i} \geq 0$. Thus, every supply order is guaranteed to be revenue adequate.
\end{proof}
Theorem \ref{individual:revenue:adequacy:supply:order} considers the surplus of supply orders, which is critical for each pricing mechanism. In the following corollary, the surplus of demand orders will also be investigated.
\begin{corollary}[Demand-Side PAO-Free Condition]\label{coro:consumer}Consider a demand order $i$ ($q_a^t(i) < 0$), the ALD pricing mechanism does not unconditionally guarantee the revenue adequacy (i.e., non-negative surplus). However, revenue adequacy holds for demand order $i$ if its bidding price $p_a^t(i)$ exceeds the offer price of any supply orders. 
\par
(A weaker yet sufficient condition for revenue adequacy of demand orders is that the bid price $p_a^t(i)$ remains higher than the offer price of any cleared supply order $i$ (i.e., $x_a^{t,*}(i) > 0$) for all $t \in \mathcal{T}$. )
\end{corollary}
\begin{proof}
For the first claim, an illustrative example is provided (see Example 1 of the Supplement), and we focus on the second claim. By replacing the objective function in \eqref{individual:obj} by $f_{a,i}^t(x_a^{t}(i))$, the optimization problem becomes\begin{equation}\label{individual:obj:sum:formulation}
\begin{aligned}
 \min&\sum_{t\in\mathcal{T}}{f_{a,i}^t(x_a^{t}(i)) },\\
  \text{s.t.}&~ x_a^{t}(i) \in X_i, 
    ~\forall~a\in \mathcal{A},t\in \mathcal{T},\lambda_a^{t,*} \text{by Def. }  \ref{def:ald:pricing}, \rho \in [\rho^*,\overline{\rho}].
\end{aligned}
\end{equation}
The same argument as the proof of Theorem \ref{incentive:compatibility:ALD} can be used to prove that $\forall~t\in\mathcal{T},x_a^{t,*}(i)$ is the optimal solution of the single-period problem:
\begin{equation}\label{individual:obj:each:time:slot}
\begin{aligned}
 \min&~{f_{a,i}^t(x_a^{t}(i)) }\\
  \text{s.t.}&~ x_a^t(i) \in X_i, 
    ~\forall~a\in \mathcal{A},\forall~t\in \mathcal{T},\lambda_a^{t,*} \text{by Def. }  \ref{def:ald:pricing}, \rho \in [\rho^*,\overline{\rho}]
\end{aligned}
\end{equation}
    Therefore, $\forall ~t\in \mathcal{T}$, if $x_{a}^{t,*}(i)>0$, then the effective clearing price inequality $LMP_a^t + \rho \geq p_a^t(i), $ holds for a supply order $i$ by Theorem \ref{individual:revenue:adequacy:supply:order}. This implies that during each time period $(t-1,t)$, the bidding prices of those cleared supply orders are covered by $LMP_a^t + \rho$. 
    Recall the choice of the dual optimal variable in Eq.~\eqref{LMP:ALD:Pricing}, for each demand order $i$, we will have $p_a^t(i)> LMP_a^t + \rho, \forall ~t\in \mathcal{T}$. The daily surplus of demand order $i$ calculated by
    \begin{equation}\label{demand:surplus:calculation}
    S_{a,i} = \sum_{t \in \mathcal{T}} |q_a^t(i)| \left( p_a^t(i) - (LMP_a^t + \rho) \right) x_a^{t,*}(i),
    \end{equation}
     is thus non-negative. Therefore, the bidding price condition is a sufficient condition for the revenue adequacy of demand orders under the ALD pricing mechanism.
\end{proof}
From Corollary \ref{coro:consumer}, it is clear that, unlike the LP model, the MILP formulation exhibits unique properties. In the LP model, all bidding orders can participate in the market, and the market-clearing price $\text{LMP}_a^t$ can ensure that there are no paradoxically accepted orders. In the MILP model, because of the power balance equation, the market operator will first address the market-clearing issue. An expensive supply order may sell electricity to a demand order with a low bidding price, which makes the demand order's surplus negative, but that negative surplus can be absorbed by another demand order that has a high bidding price, which makes the overall welfare larger than without letting the expensive supply order produce electricity. Hence, the overall welfare after the selection by the condition of Corollary \ref{coro:consumer} can be less than without the selection. If the market operator wants to put all bidding orders into the market without selection, we propose eliminating paradoxically accepted orders (PAOs) from demand orders as follows: this would put additional money into the market.
\par
In the individual optimization problem \eqref{individual:obj}, the term $\rho \sum_{t} q_a^t(i)x_a^{t,*}(i)$ serves as a non-convex reward for the supply order ($q > 0$); however, it imposes a non-convex cost on the demand order ($q < 0$). To eliminate negative surplus among demand orders, we propose a remedy approach based on targeted compensation.
Let $\mathcal{O}_{nwb}$ denote the set of ``not-well-behaved'' demand orders whose daily surplus is negative under the standard ALD price mechanism. By injecting additional revenue: a transfer payment of $\Omega_i = -2\rho \sum_{t} q_a^t(i)x_a^{t,*}(i)$ for each $i \in \mathcal{O}_{nwb}$, the final term in the reformulated individual objective of \eqref{individual:obj} becomes $-\rho \sum_{t} |q_a^t(i)x_a^{t,*}(i)|$. Consequently, for any demand order $i \in \mathcal{O}_{nwb}$, the optimization yields $f_{a,i}(x_a^*(i)) \leq f_{a,i}(0) = 0$, ensuring the daily surplus non-negative. This remedy approach, as shown by an illustrative example (Example 1 of the Supplement), can effectively restore revenue adequacy across all demand orders.
\par
We now proceed with paradoxically rejected orders (PROs), introduced by the following theorem.
\begin{theorem}[PRO-Free Property of ALD]\label{no:pro}Under the ALD pricing mechanism defined in Def. \ref{def:ald:pricing}, the market-clearing solution for model \eqref{MILP:original:one} is free of paradoxically rejected orders (PROs).
\end{theorem}
\begin{proof}
The proof follows from the optimality of the ALD-based market-clearing. Suppose, for the sake of contradiction, that an order $i$ is paradoxically rejected, i.e., $x_i^* = 0$ despite the order being " in-the-money " based on the market price. By the incentive compatibility property established in Theorem \ref{incentive:compatibility:ALD}, any deviation from the optimal dispatch $x_i^* = 0$ to $x_i > 0$ would result in a lower individual surplus for that order. Specifically, for a demand order, the marginal revenue at the given ALD price would be insufficient to cover its supply cost plus the exact penalty term. Similarly, for a demand order, the utility gained from the purchase would be outweighed by the total payment required, including the penalty. In conclusion, since the ALD mechanism ensures that the optimal dispatch maximizes each order's surplus relative to the penalty-adjusted prices, no order has a positive surplus to be accepted at the clearing price, thus eliminating PROs.
\end{proof}
Note that if the individual optimization problem \eqref{individual:obj} has only the feasible solution $0$, we would not call it a PRO, even if the market-clearing price is higher than its cost or lower than its bidding price; see the second last example of the Supplement. In addition, the last example of the Supplement will display how the ALD pricing mechanism eliminates the PROs and PAOs.
\section{The modified SAVLR method}\label{computational:method:zero:duality}
The ALD pricing framework, based on MILP strong duality, was detailed in Section \ref{ALD:pricing:proposing}. Since MILP is NP-hard, the computational burden of large-scale market clearing increases significantly with problem size. Although commercial solvers utilize advanced methods, such as branch-and-cut to solve the primal problem, computing the exact dual prices remains a clear challenge. Consequently, the primary objective of this section is to present a computationally tractable solution methodology specifically designed for the dual problem to ensure a zero duality gap.
\subsection{The related work}
Recall Lemma~\ref{proper:dual:optimization:ALD}, the dual problem~\eqref{mix:integer:linear:program:ALD} is concave and sub-differentiable. To compute the optimal dual variable, an iteration procedure is needed. However, computing the exact subgradient at each iteration step is computationally expensive, as each step requires solving a relaxed version of the original large-scale MILP. To address this, various decomposition techniques have been developed to partition the problem into smaller, easily handled sub-problems. Consequently and also in the relevant literature, a surrogate subgradient is employed in place of the exact subgradient to alleviate the computational burden.
\par
The surrogate subgradient method for MILP dual optimization was pioneered in \citep{zhao1999surrogate}, with its convergence properties established in \citep{bragin2015convergence}. Subsequent advancements introduced adaptive step sizes and re-initialization techniques \citep{bragin2016efficient} to accelerate convergence. To further enhance numerical accuracy and stability, the Surrogate Absolute-Value Lagrangian Relaxation (SAVLR) was developed \citep{bragin2018scalable} by incorporating an $L_1$-norm penalty. 
\par
Building upon the SAVLR framework, this paper proposes a modified SAVLR method to solve problem \eqref{mix:integer:linear:program:ALD}. Our approach is specifically designed to eliminate the duality gap by determining the optimal multiplier $\lambda^*$ and exact penalty coefficient $\rho \in [\rho^*, \overline{\rho}]$. While this strengthening procedure entails a higher computational cost, it ensures global optimality for a non-convex market clearing. Details will be explained in the following sections.
\subsection{The decomposed primal update}\label{sec:deco:sub:surro:sub}
In the ALD formulation \eqref{mix:integer:linear:program:ALD}, the global constraint \eqref{network:injection:equation} is relaxed but subsequently penalized by the $L_1$ norm. Since this penalty term maintains the coupling between variables, we temporarily exclude it to obtain the decomposed subproblems. 
\begin{definition}[MILP Decomposition]\label{sub:pro:definition}
We can decompose the feasible set $X$ as $X_i$ by
exploiting the separable structure in \eqref{MILP:original:one}, and the set of net injections $P$ is not separable. The decomposed feasible sets are:
\begin{enumerate}
     \item $X^t_{E,d,a,i} :=\{x_{E,d,a}^t(i)\mid x_{E,d,a}^t(i)\in X^t_{E,d,a,i}\subset \mathbb{R}^{1}\}$, the feasible set for a demand elementary order $i$ at time $t$ and area $a$.
     \item $X_{E,s,a,i} :=\{x_{E,s_k,a}(i)\mid x_{E,s,a}(i)\in X_{E,s,a,i}\subset \mathbb{R}^{|\mathcal{T}|}\}$, the feasible set of a supply elementary order $i$ at area $a$ for all time steps, (see the coupling constraint: load gradient condition \eqref{load:gradient}).
     \item $X_{BO,a,i} :=\{x_{BO,a}(i)\mid x_{BO,a}(i)\in X_{BO,a,i}\subset \mathbb{R}^1\}$, the feasible set of a block order $i$ at the area $a$ which is not coupled with others.
     \item $X_{BO,a,g-c} :=\{x_{BO,a}^{g-c}\mid x_{BO,a}^{g-c}\in X_{BO,a,g-c}\subset \mathbb{R}^{n_{g-c}}\}$, the feasible set of block orders which are coupled together at the area $a$, (see the linked constraint or exclusive group constraint \eqref{block:order:constr:2}).
     \item $X_{FHB,a,i} :=\{x_{FHB,a}(i)~|~x_{FHB,a}(i)\in X_{FHB,a,i}\subset \mathbb{R}^{|\mathcal{T}|}\}$, the flexible hourly order $i$ at the area $a$, (see the coupling constraint \eqref{block:order:constr:3}).
     \item $P :=\{p\in P \subset \mathbb{R}^{|\mathcal{A}|\times |\mathcal{T}|}\}$, the feasible set of net injection variables defined in Section \ref{ALD:to:Euphemia} explained below:
     \begin{itemize}
         \item[-] Set the last node as the reference node, i.e., $p_{|\mathcal{A}|}^t=p_{\text{ref}}^t$, $p^t = \begin{bmatrix}p_{\text{red}}^t&p_{\text{ref}}^t \end{bmatrix}$.  Then by \eqref{power:balance_flow}, $p_{\text{ref}}^t + \mathbf{1}_{|\mathcal{A}|-1}^\top p_{\text{red}}^t= 0, \forall ~t\in \mathcal{T}$ .
         \item[-] The constraints \eqref{ramping:p:hour}, \eqref{ramping:p:daily} can be formulated as $p\in P_{ramp}$, and the constraints \eqref{ramping:f:single}, \eqref{ramping:f:Line:set} can be written as $f\in F_{ramp}$. Recall the line-node incidence matrix $B= \begin{bmatrix} B_1&b\end{bmatrix}$, where $b$ corresponds to the reference node, then $f = B_1 (B_1^\top B_1)^{-1} p_{\text{red}}^t \in P_{ramp}$.
         \item[-] The constraints above can be compactly represented as $p\in P$.
     \end{itemize}
\end{enumerate}
\end{definition}
\par
Note that Def. \ref{sub:pro:definition} presents the interdependencies of the feasible region of each individual problem in Def. \ref{individual:optimization}. For the ease of notation, we still use  $X_i$ to represent an order feasible set.
During the iteration process, given $\lambda^{k+1}\in \mathbb{R}^{|\mathcal{A}|\times |\mathcal{T}|}$, in order to avoid directly solving \eqref{mix:integer:linear:program:ALD}, according to the decomposition in Def. \ref{sub:pro:definition}, one can solve the following subproblems instead.
\begin{itemize}
    \item For the subproblem subjected to an order feasible set $X_i$:
\begin{equation}\label{decision:separate:ALD}
    \begin{aligned}
     x_i^{k+1} \coloneqq \argmin_{x_i \in X_i}\tilde{L}_\rho(x_i,p^{k},\lambda^{k+1}),\\
     \tilde{L}_\rho(x_i,p^{k},\lambda^{k+1}):=~(c_i- q^\top_i\lambda^{k+1})^\top x_i + \rho \| p^k-\langle q_i, x_i\rangle - \langle q_{-i},x^k_{-i}\rangle \|_1, 
\end{aligned}
\end{equation}
where $q_{-i}$ is the quantity vector excluding the part of subproblem $i$.
\item 
For the subproblem subjected to $P$:
\begin{equation}\label{nodal:injection:ALD}
    \begin{aligned}
     p^{k+1} \coloneqq \argmin_{p \in P}\tilde{L}_\rho(x^k,p,\lambda^{k+1}),\\ \tilde{L}_\rho(x^k,p,\lambda^{k+1})
    :=~{\lambda^{k+1}}^\top p + \rho \| p-\sum_i{\langle q_i ,x_i^k\rangle }\|_1.
\end{aligned}
\end{equation}
\end{itemize}
\par
 Among the subproblems, those with feasible sets defined in 1, 2, and 6 are Linear Programming (LP), while those with feasible sets defined in 3, 4, and 5 are small-scale MILPs. Specifically, the subproblem with feasible set 4 represents the largest MILP. Given that most decomposed problems involve only a single variable, computational efficiency can be significantly improved. 
\par
\begin{definition}[The Surrogate Optimality Conditions]\label{Serro:condition}
For the subproblem~\eqref{decision:separate:ALD}, the \textbf{surrogate optimality condition} is defined as
$\tilde{L}_\rho(x_i^{k+1}, p^k, \lambda^{k+1}) < \tilde{L}_\rho(x_i^k, p^k, \lambda^{k+1}),$
and for the subproblem \eqref{nodal:injection:ALD}, the condition is
$\tilde{L}_\rho(x^k, p^{k+1}, \lambda^{k+1}) < \tilde{L}_\rho(x^k, p^k, \lambda^{k+1}).$
\end{definition}
The decomposition in Def. \ref{sub:pro:definition} and the surrogate optimality conditions in Def. \ref{Serro:condition} lead to the approach of updating primal variables in Algorithm \ref{alg:primal:update}.
  \begin{algorithm}[h!]
  \caption{The Primal Update Approach}\label{alg:primal:update}
  \begin{small}
  \begin{algorithmic}[1]
\renewcommand{\algorithmicrequire}{\textbf{Input:}}
\renewcommand{\algorithmicensure}{\textbf{Output:}}
\Require $\lambda^{k+1},\rho^{j+1}>0$ \Comment{A dual variable value, and penalty coefficient}
\Ensure $x^{k+1}, p^{k+1}$ \Comment{Numerical value of order acceptance and net injection}
        \For{$x_i \in X_i$ in Def. \ref{sub:pro:definition}}
            \State  $x_i^{k+1} \gets$ Eq. \eqref{decision:separate:ALD}
            \If{Def. \ref{Serro:condition} is satisfied}\Comment{Check surrogate optimality condition}
                \State $x^{k+1} \gets [x_i^{k+1}, x_{-i}^k], \quad p^{k+1} \gets p^k$\Comment{Primal block update}
                \State \textbf{break} 
            \Else
                \State  $p^{k+1} \gets$ Eq. \eqref{nodal:injection:ALD}\Comment{Go to net injection subproblem}
                \If{Def. \ref{Serro:condition} is satisfied}
                    \State $x^{k+1} \gets x^k, \quad p^{k+1} \gets p^{k+1}$
                \Else
                    \State Combine $X_i$ sets, reformulating \eqref{decision:separate:ALD}, and return to Step 1
                    \If{Step 11 cannot be further advanced}
                        \State $[x^{k+1}, p^{k+1}] \gets \arg\min L_{\rho^{j+1}}(x, p, \lambda^{k+1})$ \Comment{MILP \eqref{mix:integer:linear:program:ALD}}
                    \EndIf
                \EndIf
            \EndIf
        \EndFor
        \end{algorithmic}
        \end{small}
  \end{algorithm}
   \begin{algorithm}[h!]
\caption{The modified SAVLR Method}\label{alg:com:penalty_coe}
\begin{small}
\begin{algorithmic}[1]
\renewcommand{\algorithmicrequire}{\textbf{Input:}}
\renewcommand{\algorithmicensure}{\textbf{Output:}}

\Require $\lambda^1$, $\rho^1 > 0$, $\beta > 1$, and optimal primal solution $(x^*, p^*)$ of \eqref{mix:integer:linear:program}.
\Ensure $\lambda^{j+1},x^{j+1},p^{j+1}$
\State \textbf{Initialization:} 
\State \quad $(x^1, p^1) \gets \arg\min_{x,p} L_{\rho^1}(x, p, \lambda^1)$, $\hat{z}^1 \gets z^{MIP}$
\State \quad $\sigma^1 \gets \frac{\hat{z}^1 - L_{\rho^1}(\lambda^1, x^1, p^1)}{\| \langle q, x^1 \rangle - p^1 \|_1^2}$, $d^1 \gets p^1 - \langle q, x^1 \rangle, \lambda^2 = \lambda^1 + \sigma^1 d^1\quad j \gets 1$
\While{$\|L_{\rho^j}(x^{j+1}, p^{j+1}, \lambda^{j+1}) - z^{MIP}\|_1 / \|z^{MIP}\|_1 > \epsilon_b$}
    \State $k \gets 1, \quad \rho^{j+1} \gets \beta\rho^j $ \label{step:update_rho} \Comment{Penalty coefficient update}
    \While{$\|\lambda^{k+1} - \lambda^k\| > \epsilon_d$}
    \State $x^{k+1}, p^{k+1}\gets$ Algorithm \ref{alg:primal:update} with input $\lambda^{k+1}, \rho^{j+1}$.
    \State Update $\sigma^{k+1}$ as \eqref{step:length}
    \If{$k>N$}
    \If{$\sigma^{k+1}\ll   \frac{z^{MIP}- z_{\rho}^{ALD}(\lambda^k)}{\|p^{k+1}- \langle q,x^{k+1} \rangle \|_1^2} ~\textbf{or}~ \sigma^{k+1}  > \frac{z^{MIP}- z_{\rho}^{ALD}(\lambda^k)}{\|p^{k+1}- \langle q,x^{k+1} \rangle \|_1^2}$}
    \State{$\sigma^{k+1}=   \frac{z^{MIP}- z_{\rho}^{ALD}(\lambda^k)}{\|p^{k+1}- \langle q,x^{k+1} \rangle \|_1^2}$}
    \Comment{$z_{\rho}^{\text{ALD}}(\lambda^k)\gets$ MILP \eqref{mix:integer:linear:program:ALD}}
    \EndIf
    \EndIf
        \State Update $d^{k+1}$ as \eqref{dual:update} 
    \State $\lambda^{k+2} \gets \lambda^{k+1} + \sigma^{k+1} d^{k+1}$,
    \State $k\gets k+1$ \Comment{Dual update}
    \EndWhile
    \State $\lambda^{j+1} \gets \lambda^{k+1}$
    \State $(x^{j+1}, p^{j+1}) \gets \arg\min L_{\rho^{j+1}}(x, p, \lambda^{j+1})$\Comment{Outer level}
    
\EndWhile
\end{algorithmic}
\end{small}
\end{algorithm}
  For a large penalty coefficient $\rho$, the classical SAVLR method may converge to a feasible but not optimal point $\lambda'\in \mathbb{R}^{|\mathcal{A}|\times |\mathcal{T}|}$, see \citep[Prop. 1]{bragin2018scalable}. Therefore, the authors of \citep{bragin2018scalable} reduce the penalty coefficient per iteration, by $\rho^{k+1} =\rho ^k/\beta,\beta>1$, aiming at satisfying the surrogate optimality condition [Def. \ref{Serro:condition}]. In the worst case $\rho = 0$, the SAVLR method decreases to the Surrogate Lagrangian Relaxation method in \citep{bragin2015convergence}, where there are theorems ensuring the convergence to the dual optimal point. However, if the penalty coefficient is decreased, the duality gap can be enlarged, see Fig.~\ref{centralized:ALD:with:rho:one:area}, 
  which will destroy the proposed pricing mechanism. Therefore, we apply the steps introduced below instead.
   \begin{remark}[Modification of SAVLR]\label{modification:SAVLR}
      The steps 11, 12, 13 of Algorithm \ref{alg:primal:update} are the critical steps for the modification of the classical SAVLR method, which ensures that the outer level approach to be introduced in Algorithm \ref{alg:com:penalty_coe} computes the exact penalty coefficient $\rho \in [\rho^*, \overline{\rho}]$.
    \end{remark}
  \par
 By the combination technique introduced in Remark \ref{modification:SAVLR}, the distance of the difference between the surrogate subgradient $g_s = p^{k+1}-\sum_i \langle q_i,x_i^{k+1}\rangle$, and the real subgradient $g_r = p^{k+1}- \langle q,x^{k+1}\rangle$: $\|g_s -g_r\|$ decreases. Then, it is more likely that the surrogate optimality condition will be satisfied. The worst case is that all the subproblems need to be combined, i.e., solve MILP \eqref{mix:integer:linear:program:ALD} to compute the real subgradient; in that case, the decomposition loses its validity, which constitutes the primary limitation of our modified approach.  
\subsection{The dual update}
At iteration $k$, with  $\lambda^{k+1}\in \mathbb{R}^{|\mathcal{A}|\times |\mathcal{T}|}, \rho^{k+1}>0$, and primal solution $(x^{k+1}, p^{k+1})$ updated by Algorithm \ref{alg:primal:update}, the dual variable update is as follows: 
\begin{align}
      s^{k+1} &= 1-\frac{1}{{(k+1)}^{r_k}}, M_k\geq 1, r_k\in [0,1],\nonumber\\
\sigma^{k+1}&= (1-\frac{1}{M_k(k+1)^{s^{k+1}}}) \frac{\sigma^{k}\|p^{k}-\sum_i \langle q_i,x_i^{k}\rangle\|_1}{\|p^{k+1}-\sum_i \langle q_i,x_i^{k+1}\rangle \|_1},\label{step:length}\\
    d^{k+1} &= p^{k+1}-\sum_i \langle q_i,x_i^{k+1}\rangle,~
    \lambda ^{k+2} = \lambda^{k+1} + \sigma^{k+1} d^{k+1},\label{dual:update}
    \end{align}  
    where in \eqref{step:length}, the sequences $\{M_k\}$ and $\{r_k\}$ monotonically decrease to $M \geq 1$ and $r \in [0, 1]$, respectively, and the generated step length $\sigma^{k+1}$ is called an adaptive step length \citep[Sections 3.2.1]{bragin2016efficient}. An example of the parameter sequences can be $M_k = 2000/\sqrt{k} + 1, r_k = 1/\sqrt{k} + 0.01$. Before introducing the convergence theorem, we introduce the following remark.
    \begin{remark}\label{step:re:initialization}
The technique used at steps 10 and 11 of Algorithm \ref{alg:com:penalty_coe} is called Stepsize Re-initialization, \citep[Section 3.2.2]{bragin2016efficient}, which is used to avoid oscillation. This technique involves solving the centralized MILP \eqref{mix:integer:linear:program:ALD} and is applied only after sufficient iterations of the surrogate subgradient. 
\end{remark}
    The convergence theorem for the sequence of generated dual variable values is then introduced.
    \begin{theorem}[Dual Convergence]\label{convergence:modified:SAVLR}
        With the surrogate subgradient $d^{k+1} := p^{k+1}-\sum_i \langle q_i,x_i^{k+1}\rangle$ updated from Algorithm \ref{alg:primal:update}, and the sequence $\{\lambda^1,\lambda^2,\cdots\}$ generated by Algorithm \ref{alg:com:penalty_coe} converges to a unique fixed point $\lambda^*$, where $\lambda^* = \argmax\limits_{\lambda \in \mathbb{R}^{|\mathcal{A}|\times |\mathcal{T}|}} {~z_\rho^{ALD}(\lambda)}$,(see \citet[Theorem 2.1]{bragin2015convergence} and \citet[Theorem 1]{bragin2018scalable}). 
    \end{theorem}
    \par
    The preceding discussion leads to the two-level Algorithm \ref{alg:com:penalty_coe}. A similar approach for computing the exact penalty coefficient can be found in \citep[Section 4.1]{chen2010extended}. 
    Furthermore, as shown in \citep[Theorem 14]{lefebvre2024exact}, $\forall ~\lambda \in \mathbb{R}^{|\mathcal{A}|\times |\mathcal{T}|}$ the exact penalty coefficient $\rho(\lambda)$ can be computed in polynomial time. Consequently, while the centralized MILP is occasionally solved, the computational complexity of Algorithm \ref{alg:com:penalty_coe} is primarily determined by the primal MILP subproblems. By exploiting the computational efficiency of high-performance commercial MILP solvers, the proposed algorithm demonstrates strong scalability on large-scale market-clearing instances. We therefore conclude that the ALD pricing mechanism introduced in Def. \ref{def:ald:pricing} is both theoretically rigorous and computationally tractable.
\section{Illustrative examples}\label{compare:ALD:and:others}
This section compares the ALD pricing mechanism (proposed in Section \ref{ALD:pricing:proposing}) 
with existing literature benchmarks. We utilize a representative example from \citep{schiro2015convex} which, despite modeling a unit commitment problem, effectively illustrates the properties of different pricing schemes.
We compare ALD with Integer Programming (IP) pricing \citep{o2005efficient}, Semi-Lagrangian Relaxation (SLR) pricing \citep{araoz2011semi}, Primal-Dual (PD) pricing \citep{rui2012}, and Convex Hull pricing (CHP) \citep{schiro2015convex}.
\par
\par To comprehensively evaluate these mechanisms, we analyze the Lost Opportunity Cost (LOC), Individual Surplus, Total Consumer Payment, and Uplift Payment. These metrics directly correspond to the critical market properties of incentive compatibility, individual revenue adequacy, economic efficiency, and pricing transparency, respectively.
\subsection[Demonstrative examples]{A demonstrative example \citep{schiro2015convex}}
\label{sec:shiro:sexample}
\paragraph{\textbf{The 1-area case}}
The optimization problem is formulated as follows: 
\begin{equation}\label{one:area:case}
\begin{aligned}
\min~& 2500 x_1 + 500 x_2 \\
\text{s.t.}~&50 x_1 + 50 x_2 = 35,~ x_1\in [0.2, 1],~ x_2\in \{0,1\}.
\end{aligned}
\end{equation}
It is obvious that the optimal solution of \eqref{one:area:case} is $x_1 = 0.7, x_2 = 0$, and the corresponding optimal objective function value is $1750$.
\paragraph{\textbf{The 2-area case}}
The optimization problem is as follows: 
\begin{equation}\label{two:area:case}
\begin{aligned}
\min~&2500 x_1 + 500 x_2\\
\text{s.t.}~&\begin{bmatrix} p_1\\-p_1\end{bmatrix}= \begin{bmatrix}50 x_1-35\\50 x_2\end{bmatrix},\\
&~\qquad p_1 = f,~-10\leq f \leq 10,~ x_1\in [0.2, 1],~ x_2\in \{0,1\}.
\end{aligned}
\end{equation}
Evidently, the optimal solution of \eqref{two:area:case} is $x_1= 0.7,x_2 = 0,p_1 =0, f = 0$, and the corresponding optimal objective function value is $1750$.
\paragraph{\textbf{Benchmark pricing to the 1-area case}}
Given that the pricing mechanisms under comparison are well established and the illustrative example is straightforward, detailed derivations are omitted for brevity. It is noteworthy that the implementation complexity of the Primal-Dual (PD) pricing mechanism primarily stems from the linearization of the revenue adequacy constraint via the Big-M method and the associated discretization procedures. In this instance, the LP relaxation yields a price of 10 EUR/MWh. Notably, the PD price exhibits a significant departure from this LP relaxation benchmark (see \citep[abstract]{rui2012}), while closely aligning with the results of Semi-Lagrangian Relaxation (SLR). The performance metrics for each mechanism are summarized in Table \ref{each:pricing:one:area}, with the following notations: LMP (Locational Marginal Price, EUR/MWh), UP (Uplift Payment, EUR), LOC (Lost Opportunity Cost, EUR), and $S$ (Surplus, EUR).
\renewcommand{\arraystretch}{0.8}
\begin{table}[htbp]
\centering
\small
\begin{tabular}{rrrrrr}
\toprule
&$\text{LMP}$&UP($G_1$)&UP($G_2$)&LOC($G_1$)& LOC($G_2$)\\
\midrule
IP&50.00&0.00&2000.00&0.00&2000.00\\
SLR&50.00&0.00&0.00&0.00&2000.00\\
PD&50.42&0.00&0.00&0.00&2000.00 \\
CHP&10.00&1000.00&0.00&1000.00&0.00\\
ALD&10.00&0.00&0.00&0.00&0.00\\
\bottomrule
\end{tabular}\\
\begin{tabular}{rrrr}
\toprule
S($G_1$)&S($G_2$)&Total payment\\
\midrule
0.00&2000.00&3750.00\\
0.00&0.00&1750.00\\
0.00&0.00&1750.00\\
-400.00&0.00&1350.00\\
0.00&0.00&1750.00\\
\bottomrule
\end{tabular}
\caption{The results of each pricing mechanism of the 1-area case.}
\label{each:pricing:one:area}
\end{table}
\paragraph{\textbf{Benchmark pricing to the 2-area case}}
The numerical results for each pricing mechanism in the 2-area scenario are summarized in Table \ref{each:pricing:two:area}. To accommodate the multi-area setting, the Transmission System Operator (TSO) is incorporated as a critical market entity. We evaluate the TSO's financial performance by analyzing its surplus, lost opportunity cost, and uplift payment, denoted respectively as $S(\text{TSO})$, $\text{LOC}(\text{TSO})$, and $\text{UP}(\text{TSO})$ (all in EUR). Notably, the LMPs derived from the LP relaxation are $\text{LMP}_1 = 50$ EUR/MWh and $\text{LMP}_2 = 10$ EUR/MWh. These prices result in revenue inadequacy for the TSO, as the price differential incentivizes cross-area transmission while the physical flow remains zero. As illustrated in Table \ref{each:pricing:two:area}, the primal-dual pricing results in this 2-area case are close to the LP relaxation benchmarks, consistent with \citet[abstract]{rui2012}.
\paragraph{\textbf{Comparative analysis}}
The results in Table \ref{each:pricing:one:area} clearly demonstrate that the ALD pricing mechanism achieves the highest economic efficiency, as evidenced by its total payment of $1750$ EUR, the lowest among all listed pricing mechanisms. In particular, the ALD mechanism operates without any uplift payments ($\text{UP}(G_1) = \text{UP}(G_2) = 0$ EUR), ensuring that price signals remain transparent to all market orders. Furthermore, the vanishing lost opportunity costs ($\text{LOC}(G_1) = \text{LOC}(G_2) = 0$ EUR) numerically validate the incentive compatibility property established in Theorem \ref{incentive:compatibility:ALD}.
\par
Complementing the comparative analysis in Table \ref{each:pricing:one:area}, Table \ref{each:pricing:two:area} provides the TSO's computational results. Since the general performance of each pricing mechanism remains consistent across both tables, in Table \ref{each:pricing:two:area} we focus specifically on the TSO's financial position. Under PD pricing, the TSO incurs a lost opportunity cost (LOC) of 400 EUR, revealing potential incentive issues. Similarly, CHP pricing yields a 400 EUR LOC, which must be covered by uplift payments. In contrast, the ALD pricing mechanism eliminates both the LOC and the need for uplift payments, thereby ensuring incentive compatibility and an uplift-free settlement for the TSO in this scenario.
\renewcommand{\arraystretch}{0.8}
\begin{table}[!htbp]
\centering
\small

\begin{tabular}{rrrrrrr}
\toprule
&$\text{LMP}_1$&$\text{LMP}_2$&UP($G_1$)&UP($G_2$)&LOC($G_1$)& LOC($G_2$)\\
\midrule
IP&50.00&50.00&0.00&2000.00&0.00&2000.00\\
SLR&50.00&50.00&0.00&0.00&0.00&2000.00\\
PD&50.42&10.00&0.00&0.00&0.00&0.00 \\
CHP&50.00&10.00&0.00&0.00&0.00&0.00\\
ALD&25.00&25.00&0.00&0.00&0.00&0.00\\
\bottomrule
\end{tabular}
\begin{tabular}{rrrrrr}
\toprule
S($G_1$)&S($G_2$)&UP(TSO)&S(TSO)&LOC(TSO)&Total payment\\
\midrule
0.00&2000.00&0.00&0.00&0.00&3750.00\\
0.00&0.00&0.00&0.00&0.00&1750.00 \\
0.00&0.00&0.00&0.00 &400.00&1750.00\\
0.00&0.00&400.00&400.00 &400.00&2150.00\\
0.00&0.00&0.00&0.00&0.00&1750.00\\
\bottomrule
\end{tabular}
\caption{The results of each pricing mechanism of the 2-area case.}
\label{each:pricing:two:area}
\end{table}
\par
It is noteworthy that the results of each pricing mechanism of this academic example mutually confirm their properties presented in Table~\ref{tab:pricing_comparison}.
\subsection{Results of the ALD pricing mechanism} \label{Application:ALD:pricing:mechanism}
This section illustrates the performance of the ALD pricing mechanism introduced in Definition \ref{def:ald:pricing} using the benchmark case from Section \ref{sec:shiro:sexample}.
\paragraph{\textbf{The 1-area case}}
Applying the ALD approach to \eqref{one:area:case}, and reformulating the absolute value operator yields the following optimization problems:\\
\begin{equation*}
\begin{aligned}
& \min_{x_1, x_2} \quad 2500 x_1 + 500 x_2 + \lambda(35-50 x_1 -50 x_2) + \rho |35-50 x_1 -50 x_2| \\
& \text{s.t.} \quad x_1 \in [0.2, 1], \quad x_2 \in \{0, 1\}.
\end{aligned}
\end{equation*}
\begin{center}
    $\downarrow$ \text{\scriptsize Reformulation}
\end{center}
\begin{equation}\label{ALD:shiro:example:one:area}
\begin{aligned}
& \min_{x_1, x_2, q} \quad 2500 x_1 + 500 x_2 + \lambda(35-50 x_1 -50 x_2) + \rho q \\
& \text{s.t.} \quad q \geq 35-50 x_1 -50 x_2,\quad q \geq -(35-50 x_1 -50 x_2) \\
& \phantom{\text{s.t.}} \quad q\geq 0, \quad x_1 \in [0.2, 1], \quad x_2 \in \{0, 1\}.
\end{aligned}
\end{equation}
Eq.~\eqref{ALD:shiro:example:one:area} involves a bilinear expression $\lambda(35-50 x_1 -50 x_2)$. Fixing $\lambda = \lambda^* \in \mathbb{R}$, the model transforms into a simple MILP. This allows for the characterization of the dual function under different penalty coefficients $\rho$, with the results depicted in Fig. \ref{centralized:ALD:with:rho:one:area}.
\begin{figure}[htbp]
  \centering
    \subfigure[Centralized ALD]{
    \includegraphics[width=0.31\linewidth, trim={0 0 0 0}, clip]{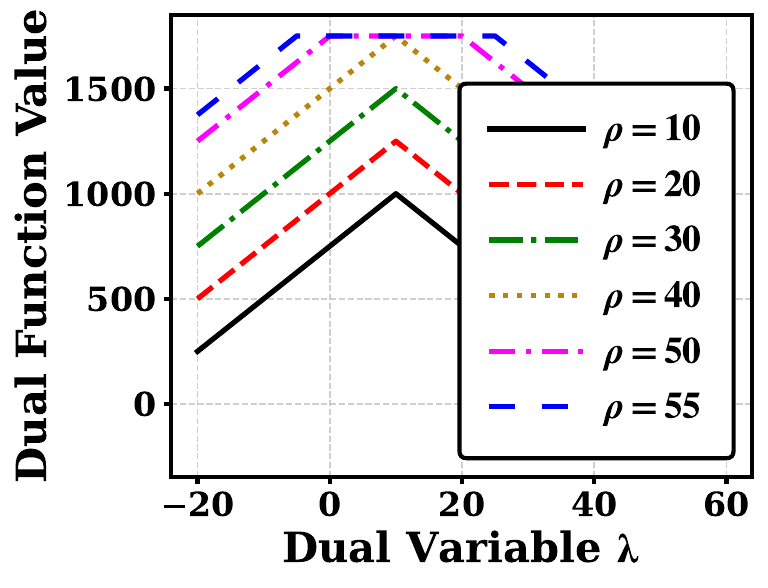}
    \label{centralized:ALD:with:rho:one:area}
  }
  \hfill
  \subfigure[ALD for generator 1]{
    \includegraphics[width=0.31\linewidth, trim={0 0 0 0}, clip]{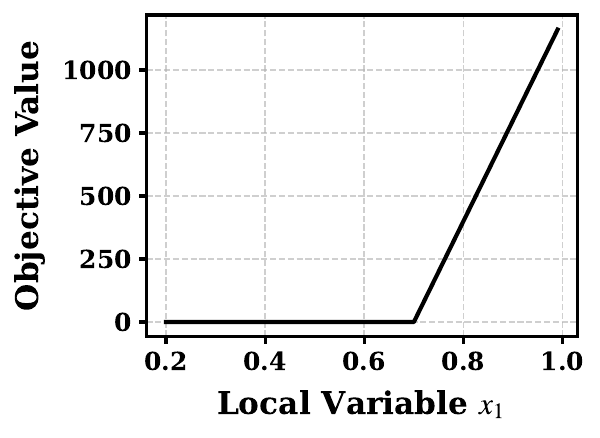}
    \label{fig:ALD:one:area:generator1}
  }\hfill
  \subfigure[ALD for generator 2]{
    \includegraphics[width=0.31\linewidth, trim={0 0 0 0}, clip]{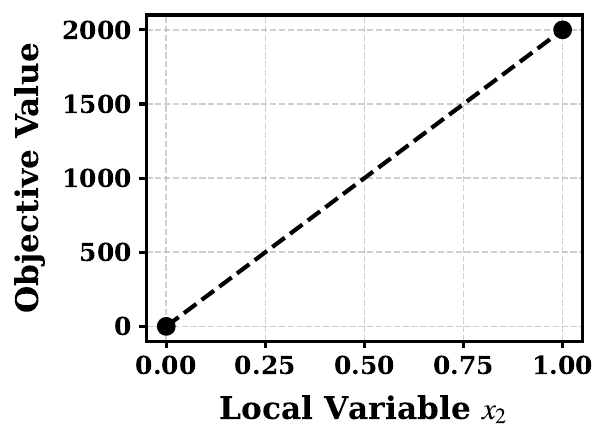}
    \label{fig:ALD:one:area:generator2}
  }
  \caption{The ALD approach to 1-area case}
  \label{fig:centralized:ALD:one:area}
\end{figure}
\par
Fig. \ref{centralized:ALD:with:rho:one:area} clearly identifies $\rho^* = 40$ as the threshold that closes the duality gap. For $\rho = \rho^*$, the unique optimal dual variable is $\lambda^* = 10$, whereas for $\rho = 45$, the optimal dual set expands to $\Lambda = [5, 15]$. Based on the ALD pricing rules in Definition \ref{def:ald:pricing}, both $(40, 10)$ and $(45, 5)$ constitute valid ALD price signals. The resulting individual optimization problems, as defined in Def. \ref{individual:optimization}, are provided below:
\begin{equation}\label{generator:1:one:area}
\begin{aligned}
&\text{For generator 1:} \min{2500 x_1  + \lambda^* (-50 x_1) + \rho |35-50 x_1| - \rho \times 50 \times x_1^*},\\
&\qquad\qquad \qquad\qquad\text{s.t.}~x_1\in [0.2, 1], x_1^* = 0.7;
\end{aligned}
\end{equation}
\begin{equation}\label{generator:2:one:area}
\begin{aligned}
&\text{For generator 2:} \min{500 x_2+ \lambda^* (-50 x_2) + \rho |-50 x_2|- \rho \times 50 \times x_2^*},\\
&\qquad\qquad \qquad\qquad \text{s.t.}~ x_2\in \{0,1\}, x_2^* = 0.
\end{aligned}
\end{equation}
The objective functions of problems \eqref{generator:1:one:area} and \eqref{generator:2:one:area} are illustrated in Fig.~\ref{fig:ALD:one:area:generator1} and Fig.~\ref{fig:ALD:one:area:generator2}, respectively. These figures depict the case with $\rho = 40$ and $\lambda^* = 10$; notably, the results for the case where $\rho = 50$ and $\lambda^* = 0$ exhibit similar characteristics. Note that the optimization problems are formulated as minimization problems; a positive objective value corresponds to a financial loss. It is clear from the observation that $x_1^* = 0.7$ and $x_2^* = 0$ are the respective optimal solutions.  Consequently, neither generator has an incentive to deviate, as any departure from these values would decrease their surplus. This observation corroborates the incentive compatibility property introduced in Theorem \ref{incentive:compatibility:ALD}. The observation that the optimal objective values are zero implies that the generators' costs are fully recovered by their revenues, thereby validating Theorem \ref{individual:revenue:adequacy:supply:order} regarding individual revenue adequacy.
\begin{figure}[t]
  \centering
  \subfigure[The dual function with $\rho = \rho^* =25$]{
    \includegraphics[width=0.46\textwidth]{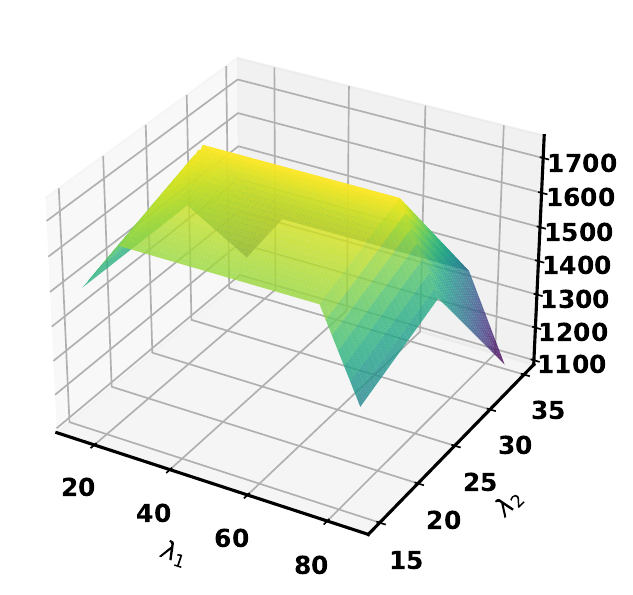}
    \label{centralized:ALD:with:rho25:two:area}
  }\hfill
  \subfigure[The dual function with $\rho =30> \rho^*$]{
    \includegraphics[width=0.46\textwidth]{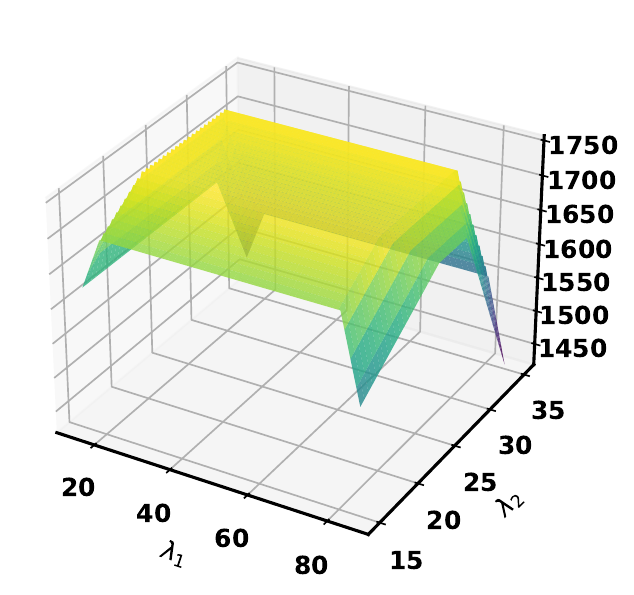}
    \label{centralized:ALD:with:rho30:two:area}
  }
  \caption{The ALD approach to the 2-area case}
  \label{centralized:ALD:with:rho30:two:area}
\end{figure}
\paragraph{\textbf{The 2-area case}}
We extend the ALD framework to the 2-area configuration \eqref{two:area:case}. The dual functions for various penalty coefficients are depicted in Fig.~\ref{centralized:ALD:with:rho30:two:area}, identifying the critical threshold $\rho^* = 25$ required to eliminate the duality gap. At this threshold, the optimal dual set is $\Lambda = \{\lambda_1 \in [25, 75], \lambda_2 = 25\}$. Increasing the penalty to $\rho = 30$ expands the set to $\Lambda = \{\lambda_1 \in [20, 80], \lambda_2 \in [20, 30]\}$. For market clearing, the operator may select $(\rho, \lambda_1^*, \lambda_2^*) = (25, 25, 25)$ as price signals. The decentralized optimization for generators follows the structure of Eq.~\eqref{generator:1:one:area}--\eqref{generator:2:one:area}, with the resulting objective landscapes aligning with Fig.~\ref{fig:ALD:one:area:generator1} and Fig.~\ref{fig:ALD:one:area:generator2}. This confirms that incentive compatibility and revenue adequacy are preserved in multi-area configurations. Furthermore, in the absence of congestion, the nodal price differential remains zero ($\lambda_1^* - \lambda_2^* = 0$) regardless of whether $(25, 25, 25)$ or $(30, 20, 20)$ is chosen. Under these signals, the TSO has no financial incentive to initiate power transfers, as the zero price yields no congestion revenue, while any deviation would incur a penalty. Consequently, the TSO also achieves revenue adequacy in this scenario. 
\par
The comprehensive results for the ALD approach are presented in Table~\ref{each:pricing:one:area} and Table~\ref{each:pricing:two:area}, with the comparative analysis provided above. 
\section{Numerical experiments}\label{computational:results:ALD:stylized:example}
In this section, we present numerical examples that apply Algorithm~\ref{alg:primal:update} and Algorithm~\ref{alg:com:penalty_coe} to demonstrate the ALD pricing mechanism. The example we employ is the Nordic day-ahead electricity market. The bidding areas and transmission lines are shown in Fig.~\ref{nordic:market}. The supply offer prices, and the demand bidding prices are produced such that each demand bidding price exceeds the largest supply offer price at each $t\in \mathcal{T}$, i.e., satisfying the sufficient condition in Corollary~\ref{coro:consumer}. To make the prices more realistic, we set the supply offer prices to gradually reach their peak values during the time periods $[7,11]$ and $[17,21]$. The capacities are configured such that the total potential supply capacity exceeds the maximum aggregate demand to ensure that the primal problem \eqref{MILP:original:one} is feasible. The number of orders in each bidding area is generated by random number generators, and Table~\ref{table:number:each:order} shows the number of different types of orders in the bidding area $10$. The ramping coefficients of network constraints are generated in accordance with the ratios of realistic energy sources, e.g., hydro power, gas, nuclear, etc., in each Nordic country. Additionally, we use the Fixed Binary-Variable (FBV) pricing method presented in \citep[Section 5]{chatzigiannis2016european} as a comparison. Distinct from the traditional Integer Programming (IP) pricing \citep{o2005efficient}, our approach follows the logic of \citep[Section 5]{chatzigiannis2016european}. We refer to this as Fixed Binary-Variable Pricing, in which the binary commitment status is fixed at its optimal value in the final price-setting stage. This ensures that the resulting LMPs reflect the marginal cost without including uplift payments.
\par
It is worth noting that the fundamental properties of major pricing mechanisms have been rigorously compared in Section \ref{compare:ALD:and:others} using an academic case. In this stylized study, further comparison with other mechanisms would primarily entail complex implementation, such as characterizing the convex hulls of feasible regions under exclusive-group, linked-block, and flexible hourly constraints, without yielding additional theoretical insights. Consequently, Fixed Binary-Variable pricing is selected as the representative benchmark for this large-scale computation. Moreover, all numerical simulations are executed in the Pyomo modeling environment using the Gurobi solver. All computations are performed on a machine equipped with an Apple M4 Pro chip and 24 GB of RAM.\\
\begin{figure}[h]
\centering
\begin{tikzpicture}[
    scale=0.8, transform shape, 
    box/.style={rectangle, draw, rounded corners=2pt, minimum width=1.4cm, minimum height=0.6cm, 
                inner sep=2pt, font=\small, align=center},
    NO/.style={box, fill=black!30},   
    SE/.style={box, fill=blue!20},    
    FI/.style={box, fill=green!20},   
    DK/.style={box, fill=red!20},     
    edge label/.style={font=\footnotesize\bfseries, inner sep=1.5pt} 
]

    \node[NO] (NO4) at (0,5)     {NO4(6)};
    
    \node[NO] (NO3) at (-2.5,3)  {NO3(5)};
    \node[SE] (SE2) at (0,3)     {SE2(9)};
    \node[SE] (SE1) at (3,4)     {SE1(8)};
    \node[FI] (FI)  at (4,2.5)   {FI(2)};
    
    \node[NO] (NO5) at (-4,1)    {NO5(7)};
    \node[NO] (NO1) at (-1.5,1)  {NO1(3)};
    \node[SE] (SE3) at (1.5,0.5) {SE3(10)};
    
    \node[NO] (NO2) at (-2.5,-1) {NO2(4)};
    \node[SE] (SE4) at (3,-1)    {SE4(11)};
    
    \node[DK] (DK2) at (1,-2.5)  {DK2(1)};
    \node[DK] (DK1) at (-1,-3.5) {DK1(0)};

    \begin{scope}[ultra thick]
        \draw (NO4) -- node[edge label, above] {$l_{11}$} (SE1);
        \draw (NO4) -- node[edge label, left=6pt]  {$l_6$} (NO3);
        \draw (NO4) -- node[edge label, right=4pt] {$l_{12}$} (SE2);
        
        \draw (NO3) -- node[edge label, above=2pt] {$l_{13}$} (SE2);
        \draw (NO3) -- node[edge label, below left] {$l_7$} (NO1);
        
        \draw (SE1) -- node[edge label, above = 2pt ] {$l_5$} (SE2);
        \draw (SE1) -- node[edge label, right = 2pt] {$l_{15}$} (FI);
        
        \draw (SE2) -- node[edge label, left=3pt] {$l_4$} (SE3);
        \draw (FI)  -- node[edge label, above left = 2pt] {$l_{14}$} (SE3);
        
        \draw (NO5) -- node[edge label, above] {$l_8$} (NO1);
        \draw (NO5) -- node[edge label, below left] {$l_{10}$} (NO2);
        \draw (NO1) -- node[edge label, below right = 2pt] {$l_9$} (NO2);
        \draw (NO1) -- node[edge label, above] {$l_{16}$} (SE3);
        
        \draw (SE3) -- node[edge label, left = 4pt] {$l_1$} (DK1);
        \draw (SE3) -- node[edge label, above right= 2pt] {$l_3$} (SE4);
        
        \draw (SE4) -- node[edge label, below right = 2pt] {$l_2$} (DK2);
        \draw (DK2) -- node[edge label, below right = 2pt]  {$l_0$} (DK1);
    \end{scope}

\end{tikzpicture}
\caption{The Nordic day-ahead electricity market topology.}
\label{nordic:market}
\end{figure}
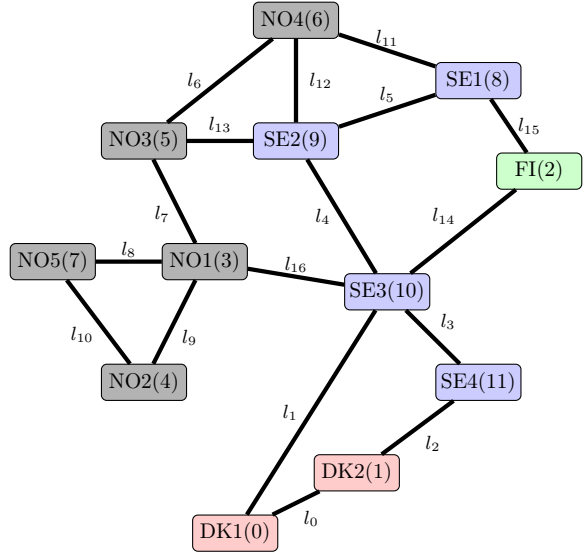\\
\begin{table}[htbp]
\centering
\begin{tabular}{lccccccc}
\toprule
$n_{E,d}$&$n_{E,s}$&$n_{PB,d}$&$n_{PB,s}$& $n_{RB,d}$& $n_{RB,s}$& $n_{FHB,d}$& $n_{FHB,s}$\\
\midrule
23 & 12 & 6&6 &6& 8&5&3\\
\bottomrule
\end{tabular}
\caption{The number of each order in bidding area 10 (SE3).}
\label{table:number:each:order}
\end{table}
\subsection{A stylized computation of the Nordic electricity market}\label{stylized:area:nordic:elec:market}
The network topology, shown in Fig.~\ref{nordic:market}, consists of $|\mathcal{A}| = 12$ nodes and $|L| = 17$ lines over $|\mathcal{T}| = 24$ time steps. The resulting mathematical model is substantial, containing 9,640 variables, including 2,248 binary decision variables per instance. Despite this complexity, sensitivity analysis using Algorithm~\ref{alg:com:penalty_coe} with different initial settings $\lambda$ and $\rho$  demonstrated robust performance, with a maximum computation time of 40 minutes.
In addition, under the sufficient condition in Corollary \ref{coro:consumer}, demand orders cannot be PROs, as their willingness to pay is sufficiently high. Any rejection of such orders would stem from structural constraints, such as linked-block constraint, etc. Consequently, potential PROs are limited to the supply side.
\par
\paragraph{\textbf{Order Surplus}} Table \ref{eleminate:PROS:12:bidding:area} presents order surpluses for potential PROs under both pricing mechanisms. The results indicate that the ALD mechanism effectively eliminates all PROs by rendering deviations unprofitable; for instance, deviations that previously appeared attractive (e.g., surpluses of 23,628.62 EUR and 55,188.95 EUR) result in significant losses under ALD (-320,933.42 EUR and -288,976.80 EUR, respectively). Similarly, as shown in Table \ref{eleminate:PAOS:12:bidding:area}, the ALD approach resolves all Paradoxically Accepted Orders (PAOs) inherent to the Fixed Binary-Variable (FBV) method. Specifically, a supply regular block order in Area 1 transits from a deficit of -241,980.01 EUR to a positive surplus of 14,553.13 EUR, while a flexible hourly order in Area 5 improves from -65,626.68 EUR to 45,515.78 EUR.
\par
\paragraph{\textbf{Congestion Rent}} Table \ref{tab:congestion_rent_comparison} details the congestion revenue of each transmission line. The aggregate TSO revenue under ALD pricing reaches 71,460.28 EUR, larger than the 3,120.43 EUR generated by the Fixed Binary-Variable method. This positive total revenue observation implies that for a general network structure, it can be revenue adequate for the TSO under the ALD pricing mechanism. It is noteworthy that ALD does not guarantee non-negative congestion rents on every individual line; for instance, Lines 3 and 5 exhibit negative rents of -83,937.70 EUR and -35,143.99 EUR, respectively. Finally, a structural analysis of Tables \ref{eleminate:PROS:12:bidding:area} and \ref{eleminate:PAOS:12:bidding:area} reveals that under Fixed Binary-Variable pricing, PROs consist exclusively of supply profile block orders, whereas PAOs are composed entirely of supply regular block orders and flexible hourly orders.
 \par
 \paragraph{\textbf{LMP and Congestion Price (CP) Analysis}} Fig.~\ref{lmp:four:areas} and Fig.~\ref{congestion:four:lines} display LMPs of four representative bidding areas and the congestion prices of four transmission lines, respectively. Under the ALD mechanism, the effective settlement LMPs are calculated as $\lambda_a^{t,*} + \rho$, with the penalty parameter set to $\rho = 30.58$ (where $\rho \in [\rho^*, \overline{\rho}]$). Congestion prices are defined based on the assumed flow directions: DK2$\rightarrow$SE4, NO1$\rightarrow$NO5, NO4$\rightarrow$SE1, and SE3$\rightarrow$FI. Both figures reveal a temporal pattern, with peak LMPs and congestion prices appearing in the time intervals [7, 11] and [17, 21]. As illustrated in Fig.~\ref{lmp:four:areas}, ALD-derived LMPs are generally higher than those from the Fixed Binary-Variable (FBV) pricing. However, as evidenced in Fig.~\ref{congestion:four:lines}, no consistent dominance relationship exists between the two mechanisms regarding congestion prices.
 \paragraph{\textbf{Welfare Analysis}} Finally, the aggregate net surplus across all orders is calculated as 111,539,574.81 EUR for the FBV method and 111,471,234.95 EUR for the ALD method. The Total Social Welfare, defined as the sum of the total order surplus and the TSO's congestion revenue, is
\[
\small 
\setlength{\arraycolsep}{2pt} 
\renewcommand{\arraystretch}{1.2}
\begin{array}{l r c c l} %
   \text{Fixed Binary-Variable:} & 111,539,574.81 &+& 3,120.43 &= \mathbf{111,542,695.24} (\text{EUR})\\
   \text{ALD:} & 111,471,234.95 &+& 71,460.28 &= \mathbf{111,542,695.24}(\text{EUR})
\end{array}
\]
The identical social welfare demonstrates no welfare loss, attributing to the sufficient condition in Corollary \ref{coro:consumer}. Significantly, these computational results validate Proposition \ref{revenue:tso}, confirming that the summation of individual surplus functions \eqref{individual:obj} characterizes the revenue allocated to the TSO.
\begin{figure}[htbp]
  \centering
  \begin{minipage}{0.48\textwidth}
    \centering
    \includegraphics[width=\linewidth]{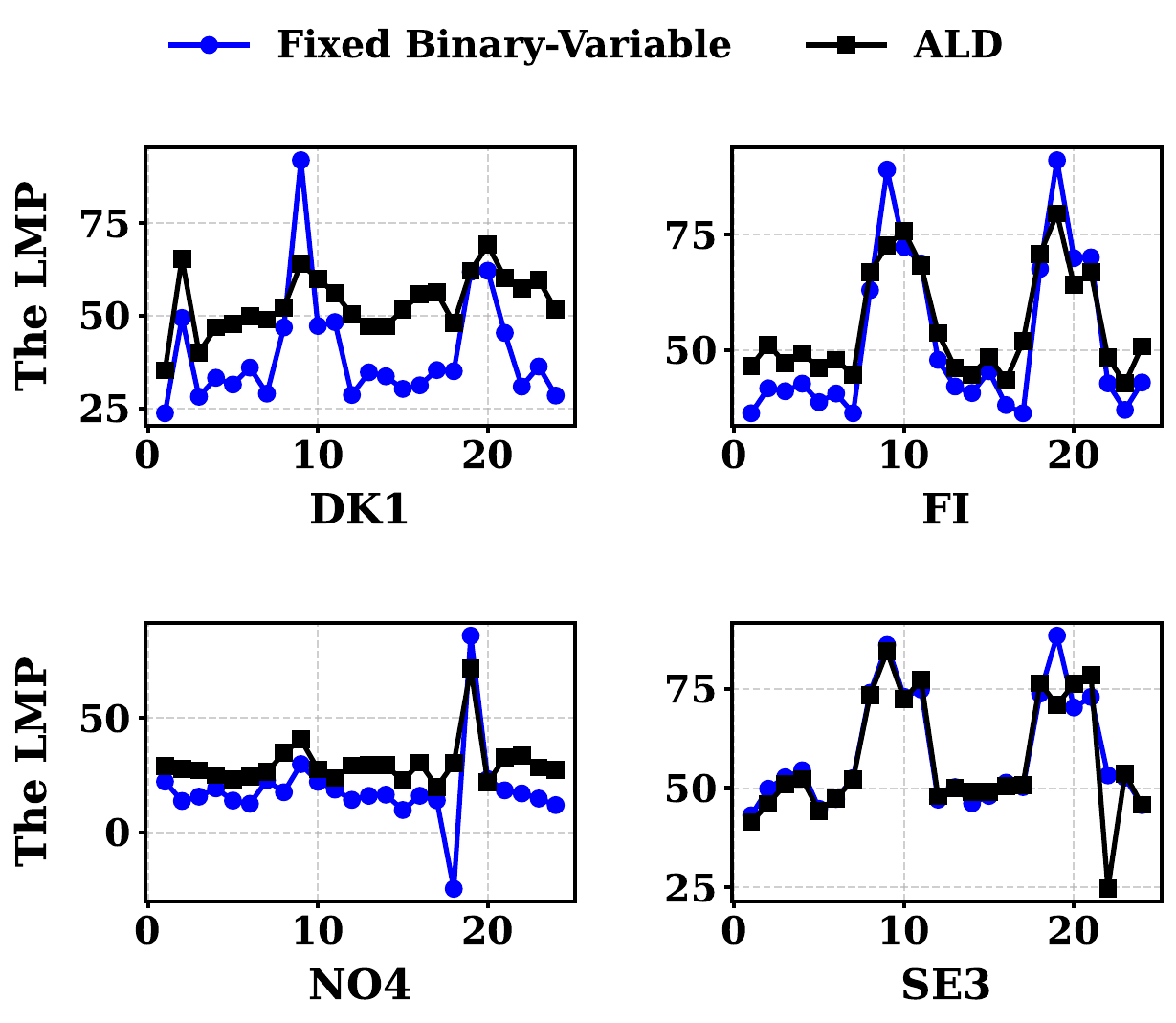}
    \caption{LMPs of DK1 FI NO4, and SE3.}
    \label{lmp:four:areas}
  \end{minipage}
  \hfill 
  \begin{minipage}{0.48\textwidth}
    \centering
    \includegraphics[width=\linewidth]{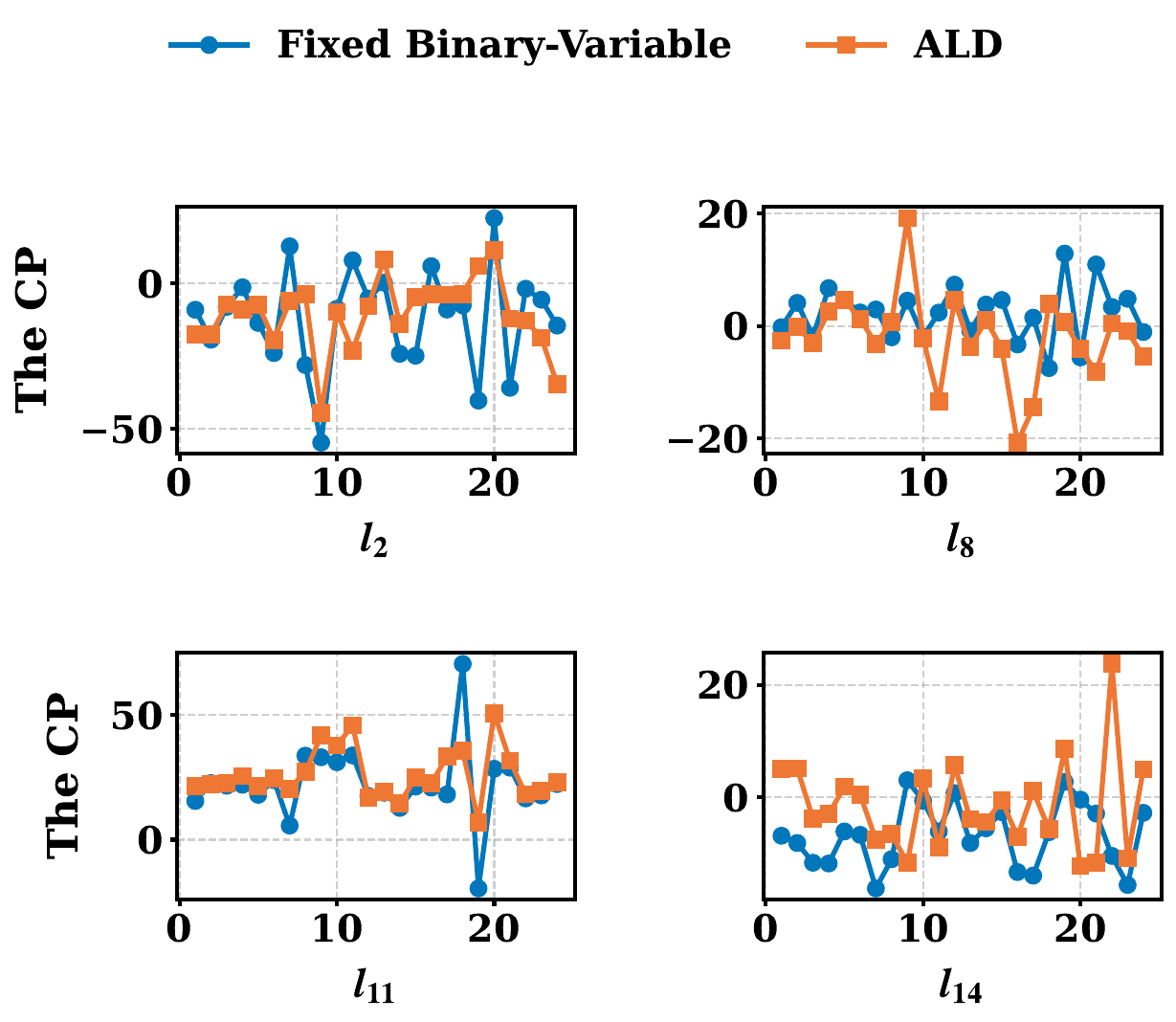}
    \caption{Congestion prices of $l_2$ $l_8$ $l_{11}$, and $l_{14}$.}
    \label{congestion:four:lines}
  \end{minipage}
\end{figure}
\begin{table}[p]
  \centering
  \caption{Surplus of Potential PROs under Deviation across the Two Pricing Mechanisms (values in EUR).}
  \label{eleminate:PROS:12:bidding:area}
  \begin{minipage}[b]{0.48\textwidth}
    \centering
    \begin{tabular}{llrr}
      \toprule
      Area & Type & ALD & \makecell[r]{Fixed \\ Binary-Variable} \\
      \midrule
      1 & PBs & -320933.42 & 23628.62 \\
      1 & PBs & -311859.08 & 20116.34 \\
      1 & PBs & -330400.36 & 29700.51 \\
      1 & PBs & -311803.45 & 21623.49 \\
      1 & PBs & -338078.87 & 31933.73 \\
      1 & PBs & -281835.01 & 29642.45 \\
      1 & PBs & -274894.84 & 34316.63 \\
      2 & PBs & -293594.67 & 46974.28 \\
      2 & PBs & -304967.75 & 66581.78 \\
      2 & PBs & -296806.04 & 61679.44 \\
      2 & PBs & -291870.20 & 59921.81 \\
      2 & PBs & -288976.80 & 55188.95 \\
      6 & PBs & -305562.97 & 9211.71 \\
      \bottomrule
    \end{tabular}
  \end{minipage}
  \hfill 
  \begin{minipage}[b]{0.48\textwidth}
    \centering
    \begin{tabular}{llrr}
      \toprule
      Area & Type & ALD & \makecell[r]{Fixed \\ Binary-Variable} \\
      \midrule
      6  & PBs & -292515.60 & 5584.97 \\
      7  & PBs & -341547.04 & 1461.39 \\
      10 & PBs & -178153.35 & 200758.35 \\
      10 & PBs & -194856.67 & 185135.18 \\
      10 & PBs & -187285.31 & 207090.16 \\
      10 & PBs & -195764.04 & 197509.82 \\
      10 & PBs & -201818.50 & 190176.82 \\
      10 & PBs & -177305.86 & 161165.21 \\
      11 & PBs & -303860.11 & 24955.95 \\
      11 & PBs & -332687.96 & 39930.19 \\
      11 & PBs & -307721.71 & 30519.63 \\
      11 & PBs & -286929.02 & 39001.60 \\
      11 & PBs & -290551.67 & 33186.83 \\
      \bottomrule
    \end{tabular}
  \end{minipage}
\end{table}
\begin{table}[htbp]
  \centering
  \caption{Surplus of PAOs under the Two Pricing Mechanisms (values in EUR).}
  \label{eleminate:PAOS:12:bidding:area}
  \begin{minipage}[b]{0.48\textwidth}
    \centering
    \begin{tabular}{llrr}
      \toprule
      Area & Type & ALD & \makecell[r]{Fixed \\ Binary-Variable} \\
      \midrule
0 &RBs& 14553.13& -241980.01\\
0&RBs &46802.42& -186208.88\\
0& RBs &37922.19& -217987.58\\
1 &RBs &69378.67& -103914.42\\
1& RBs &103046.33& -63429.62\\
1 &RBs &71680.03& -92404.38\\
2 &FHBs &21672.06& -16827.11\\
3& PBs &82489.51& -18668.97\\
3& RBs &214432.84& -143719.35\\
3& RBs &205945.53& -103015.72\\
3& RBs &192840.15& -131127.68\\
3& RBs &249548.58& -67867.13\\
3& FHBs &39548.60& -8649.98\\
4& RBs  &94217.03&-134868.43\\
4& RBs  &153001.69&-91859.72\\
      \bottomrule
    \end{tabular}
  \end{minipage}
  \hfill 
  \begin{minipage}[b]{0.48\textwidth}
    \centering
    \begin{tabular}{llrr}
      \toprule
      Area & Type & ALD & \makecell[r]{Fixed \\ Binary-Variable} \\
      \midrule
4& FHBs &58413.02& -11775.41\\
5& RBs  &311282.16&-78148.65\\
5& FHBs  &45515.78&-65626.68\\
6& RBs  &176777.06&-65749.15\\
6& FHBs  &45271.51&-6579.98\\
7& PBs  &72123.80&-4861.81\\
7& RBs  &155260.26&-92365.57\\
7& RBs  &173642.32&-86077.40\\
8& RBs & 93790.23&-210338.61\\
8& RBs  &51920.53&-224680.03\\
9& RBs  &162331.97&-75942.15\\
9& FHBs  &35304.13&-8964.91\\
11& RBs  &103812.06&-31966.78\\
11& RBs  &140318.61&-7339.24\\
11& FHBs & 25994.79&-21321.93\\
      \bottomrule
    \end{tabular}
  \end{minipage}
\end{table}
\begin{table}[htbp]
  \centering
  \caption{Congestion Rents across Lines under the Two Pricing Mechanisms (values in EUR).}
  \label{tab:congestion_rent_comparison}
  \begin{minipage}[b]{0.48\textwidth}
    \centering
    \begin{tabular}{lrr}
      \toprule
      Line & ALD & \makecell[r]{Fixed \\ Binary-Variable} \\
      \midrule
      0 & 8166.97 & 20368.84 \\
      1 & -152489.56 & -92750.42 \\
      2 & 74401.55 & 74936.62 \\
      3 & -83937.70 & -83308.50 \\
      4 & 124440.82 & 110795.22 \\
      5 & -35143.99 & -39330.08 \\
      6 & 22205.46 & 31606.41 \\
      7 & -3698.31 & 2307.88 \\
      \bottomrule
    \end{tabular}
  \end{minipage}
  \hfill %
  \begin{minipage}[b]{0.48\textwidth}
    \centering
    \begin{tabular}{lrr}
      \toprule
      Line & ALD & \makecell[r]{Fixed \\ Binary-Variable}\\
      \midrule
      8 & 21910.85 & 3972.68 \\
      9 & 3663.32 & -5400.03 \\
      10 & 8829.46 & 15065.89 \\
      11 & -79154.66 & -84060.20 \\
      12 & 1391.44 & -2152.80 \\
      13 & 4217.57 & 24297.86 \\
      14 & 62495.10 & 12103.53 \\
      15 & -31221.42 & 6754.30 \\
      16 & 57043.53 & 76253.08 \\
      \bottomrule
    \end{tabular}
  \end{minipage}
\end{table}
\section{Conclusion and the future work}\label{conclusion:future:work}
This paper proposes a novel pricing mechanism, termed the Augmented Lagrangian and Duality (ALD) pricing mechanism, derived from the strong duality framework of Mixed-Integer Linear Programming (MILP). We demonstrate that once the primal social welfare maximization problem is solved to optimality, the ALD pricing mechanism effectively eliminates all Paradoxically Rejected Orders (PROs) and Paradoxically Accepted Orders (PAOs) among supply orders. Furthermore, we establish that PAOs among demand orders can be eliminated provided that demand bid prices exceed the maximal supply offer prices. These properties are supported by rigorous theoretical proofs and are benchmarked with major pricing mechanisms in the literature to highlight the advantages of the ALD pricing approach. Finally, we have modified the SAVLR algorithm \citep{bragin2018scalable} to compute the exact penalty coefficients required to close the duality gap and the optimal dual variables. This modification ensures that the proposed ALD pricing mechanism is computationally tractable, as validated by the computational results.
\par
To further extend this research, several promising trajectories are identified: 1) Model Generalization: While this study formulates the market-clearing problem as an MILP, incorporating more complex features could lead to Mixed-Integer Quadratic Programming (MIQP) \citep{gu2020exact} or even Mixed-Integer Nonlinear Programming (MINLP)  formulations \citep{lefebvre2024exact}. Applying the proposed ALD framework to these more sophisticated models presents foreseeable challenges and a promising area for future exploration. 2) Algorithm Acceleration for 15-Minute Dispatch: In response to the European market's transition towards a 15-minute time interval, the number of integer variables will increase significantly. To maintain strict market-clearing deadlines under this resolution, future work will focus on augmenting the proposed modified SAVLR algorithm with advanced primal heuristics and parallel computing architectures to further accelerate the convergence of the primal combinatorial problem.
\section*{Acknowledgments}
    The authors thank the Swedish Energy Agency for the financial support of this research work (grant number: P2022-00738), the advisors from the Nord Pool, and the PhD students at KTH, energy market research group, for their useful comments.
    \par
    During the preparation of this work, Zhen Wang used Google Gemini in order to edit the language. After using this tool, the co-authors reviewed and edited the content as needed. We take full responsibility for the content of the published article.
    \par
 For reproducibility and further research, the ALD pricing implementation, modified SAVLR algorithm, and Nordic test cases are available on GitHub under the MIT license: \url{https://github.com/Zhen-Wang-Sudo/nordic-electricity-ald-codes}.
\bibliographystyle{elsarticle-harv}
\bibliography{reference}

@article{o2005efficient,
  title={Efficient market-clearing prices in markets with nonconvexities},
  author={O'Neill, Richard P and Sotkiewicz, Paul M and Hobbs, Benjamin F and Rothkopf, Michael H and Stewart Jr, William R},
  journal={European Journal of Operational Research},
  doi={https://doi.org/10.1016/j.ejor.2003.12.011},
  volume={164},
  number={1},
  pages={269--285},
  year={2005},
  publisher={Elsevier}
}

@article{feizollahi2017exact,
  title={Exact augmented Lagrangian duality for mixed integer linear programming},
  author={Feizollahi, Mohammad Javad and Ahmed, Shabbir and Sun, Andy},
  journal={Mathematical Programming},
  doi={https://doi.org/10.1007/s10107-016-1012-8},
  volume={161},
  pages={365--387},
  year={2017},
  publisher={Springer}
}

@article{bragin2015convergence,
  title={Convergence of the surrogate Lagrangian relaxation method},
  author={Bragin, Mikhail A and Luh, Peter B and Yan, Joseph H and Yu, Nanpeng and Stern, Gary A},
  journal={Journal of Optimization Theory and applications},
  doi = {https://doi.org/10.1007/s10957-014-0561-3},
  volume={164},
  pages={173--201},
  year={2015},
  publisher={Springer}
}

@article{zhao1999surrogate,
  title={Surrogate gradient algorithm for Lagrangian relaxation},
  author={Zhao, Xing and Luh, Peter B and Wang, Jihua},
  journal={Journal of optimization Theory and Applications},
  doi = {https://doi.org/10.1023/A:1022646725208},
  volume={100},
  pages={699--712},
  year={1999},
  publisher={Springer}
}

@article{bragin2018scalable,
  title={A scalable solution methodology for mixed-integer linear programming problems arising in automation},
  author={Bragin, Mikhail A and Luh, Peter B and Yan, Bing and Sun, Xiaorong},
  journal={IEEE Transactions on Automation Science and Engineering},
  doi = {10.1109/TASE.2018.2835298},
  volume={16},
  number={2},
  pages={531--541},
  year={2018},
  publisher={IEEE}
}

@article{chatzigiannis2016european,
  title={European day-ahead electricity market clearing model},
  author={Chatzigiannis, Dimitris I and Dourbois, Grigoris A and Biskas, Pandelis N and Bakirtzis, Anastasios G},
  journal={Electric Power Systems Research},
  doi = {https://doi.org/10.1016/j.epsr.2016.06.019},
  volume={140},
  pages={225--239},
  year={2016},
  publisher={Elsevier}
}

@article{bragin2016efficient,
  title={An efficient approach for solving mixed-integer programming problems under the monotonic condition},
  author={Bragin, Mikhail A and Luh, Peter B and Yan, Joseph H and Stern, Gary A},
  journal={Journal of Control and Decision},
  doi = {https://doi.org/10.1080/23307706.2015.1129916},
  volume={3},
  number={1},
  pages={44--67},
  year={2016},
  publisher={Taylor \& Francis}
}

@article{lefebvre2024exact,
  title={Exact Augmented Lagrangian Duality for Nonconvex Mixed-Integer Nonlinear Optimization},
  author={Lefebvre, Henri and Schmidt, Martin},
  year={2024},
journal = {optimization-online.org}
}

@article{gu2020exact,
  title={Exact augmented Lagrangian duality for mixed integer quadratic programming},
  author={Gu, Xiaoyi and Ahmed, Shabbir and Dey, Santanu S},
  journal={SIAM Journal on Optimization},
  doi = {https://doi.org/10.1137/19M127169},
  volume={30},
  number={1},
  pages={781--797},
  year={2020},
  publisher={SIAM}
}

@article{azizan2020optimal,
  title={Optimal pricing in markets with nonconvex costs},
  author={Azizan, Navid and Su, Yu and Dvijotham, Krishnamurthy and Wierman, Adam},
  journal={Operations Research},
  doi = {https://doi.org/10.1287/opre.2019.1900},
  volume={68},
  number={2},
  pages={480--496},
  year={2020},
  publisher={INFORMS}
}

@article{araoz2011semi,
  title={{Semi-Lagrangian} approach for price discovery in markets with non-convexities},
  author={Araoz, Veronica and J{\"o}rnsten, Kurt},
  journal={European Journal of Operational Research},
  doi = {https://doi.org/10.1016/j.ejor.2011.05.009},
  volume={214},
  number={2},
  pages={411--417},
  year={2011},
  publisher={Elsevier}
}

@ARTICLE{rui2012,
  author={Ruiz, Carlos and Conejo, Antonio J. and Gabriel, Steven A.},
  journal={IEEE Transactions on Power Systems}, 
  title={Pricing Non-Convexities in an Electricity Pool}, 
  year={2012},
  volume={27},
  number={3},
  pages={1334-1342},
  doi={10.1109/TPWRS.2012.2184562}}

@article{schiro2015convex,
  title={Convex hull pricing in electricity markets: Formulation, analysis, and implementation challenges},
  author={Schiro, Dane A and Zheng, Tongxin and Zhao, Feng and Litvinov, Eugene},
  journal={IEEE Transactions on Power Systems},
  doi = {10.1109/TPWRS.2015.2486380},
  volume={31},
  number={5},
  pages={4068--4075},
  year={2016},
  publisher={IEEE}
}

@article{guo2025copositive,
  title={Copositive duality for discrete energy markets},
  author={Guo, Cheng and Bodur, Merve and Taylor, Joshua A},
  journal={Management Science},
  doi = {https://doi.org/10.1287/mnsc.2023.00906},
  year={2025},
  publisher={INFORMS}
}

@article{chen2010extended,
  title={Extended duality for nonlinear programming},
  author={Chen, Yixin and Chen, Minmin},
  journal={Computational optimization and applications},
  doi = {https://doi.org/10.1007/s10589-008-9208-3},
  volume={47},
  number={1},
  pages={33--59},
  year={2010},
  publisher={Springer}
}

@article{madani2018convex,
  title={{Convex hull, IP and European electricity pricing in a European power exchanges setting with efficient computation of convex hull prices}},
  author={Madani, Mehdi and Ruiz, Carlos and Siddiqui, Sauleh and Van Vyve, Mathieu},
  journal={arXiv preprint arXiv:1804.00048},
  doi = {https://doi.org/10.48550/arXiv.1804.00048},
  year={2018}
}

@article{madani2015computationally,
  title={{Computationally efficient MIP formulation and algorithms for European day-ahead electricity market auctions}},
  author={Madani, Mehdi and Van Vyve, Mathieu},
  journal={European Journal of Operational Research},
  doi = {https://doi.org/10.1016/j.ejor.2014.09.060},
  volume={242},
  number={2},
  pages={580--593},
  year={2015},
  publisher={Elsevier}
}

@article{gribik2007market,
  title={Market-clearing electricity prices and energy uplift},
  author={Gribik, Paul R and Hogan, William W and Pope, Susan L and others},
  journal={Cambridge, MA},
  pages={1--46},
  year={2007}
}

@article{bjorndal2008equilibrium,
  title={Equilibrium prices supported by dual price functions in markets with non-convexities},
  author={Bj{\o}rndal, Mette and J{\"o}rnsten, Kurt},
  journal={European Journal of Operational Research},
  doi = {https://doi.org/10.1016/j.ejor.2007.06.050},
  volume={190},
  number={3},
  pages={768--789},
  year={2008},
  publisher={Elsevier}
}

@article{toczylowski2009new,
  title={A new pricing scheme for a multi-period pool-based electricity auction},
  author={Toczy{\l}owski, Eugeniusz and Zoltowska, Izabela},
  journal={European Journal of Operational Research},
  doi = {https://doi.org/10.1016/j.ejor.2007.12.048},
  volume={197},
  number={3},
  pages={1051--1062},
  year={2009},
  publisher={Elsevier}
}

@article{ahunbay2025pricing,
  title={Pricing optimal outcomes in coupled and non-convex markets: Theory and applications to electricity markets},
  author={Ahunbay, Mete {\c{S}}eref and Bichler, Martin and Kn{\"o}rr, Johannes},
  journal={Operations Research},
  doi = {https://doi.org/10.1287/opre.2023.0401},
  volume={73},
  number={1},
  pages={178--193},
  year={2025},
  publisher={INFORMS}
}

@article{ahunbay2024solving,
  title={Solving large-scale electricity market pricing problems in polynomial time},
  author={Ahunbay, Mete {\c{S}}eref and Bichler, Martin and Dobos, Teodora and Kn{\"o}rr, Johannes},
  journal={European Journal of Operational Research},
  doi = {https://doi.org/10.1016/j.ejor.2024.05.020},
  volume={318},
  number={2},
  pages={605--617},
  year={2024},
  publisher={Elsevier}
}

@article{byers2023long,
  title={Long-run optimal pricing in electricity markets with non-convex costs},
  author={Byers, Conleigh and Hug, Gabriela},
  journal={European Journal of Operational Research},
  doi = {https://doi.org/10.1016/j.ejor.2022.07.052},
  volume={307},
  number={1},
  pages={351--363},
  year={2023},
  publisher={Elsevier}
}

@article{byers2022economic,
  title={Economic impacts of near-optimal solutions with non-convex pricing},
  author={Byers, Conleigh and Hug, Gabriela},
  journal={Electric Power Systems Research},
  doi = {https://doi.org/10.1016/j.epsr.2022.108287},
  volume={211},
  pages={108287},
  year={2022},
  publisher={Elsevier}
}

@article{andrianesis2021computation,
  title={Computation of convex hull prices in electricity markets with non-convexities using dantzig-wolfe decomposition},
  author={Andrianesis, Panagiotis and Bertsimas, Dimitris and Caramanis, Michael C and Hogan, William W},
  journal={IEEE Transactions on Power Systems},
  doi = {10.1109/TPWRS.2021.3122000},
  volume={37},
  number={4},
  pages={2578--2589},
  year={2021},
  publisher={IEEE}
}

@article{herrero2015electricity,
  title={Electricity market-clearing prices and investment incentives: The role of pricing rules},
  author={Herrero, Ignacio and Rodilla, Pablo and Batlle, Carlos},
  journal={Energy Economics},
  doi = {https://doi.org/10.1016/j.eneco.2014.10.024},
  volume={47},
  pages={42--51},
  year={2015},
  publisher={Elsevier}
}

@article{yang2019unified,
  title={A unified approach to pricing under nonconvexity},
  author={Yang, Zhifang and Zheng, Tongxin and Yu, Juan and Xie, Kaigui},
  journal={IEEE Transactions on Power Systems},
  doi = {10.1109/TPWRS.2019.2911419},
  volume={34},
  number={5},
  pages={3417--3427},
  year={2019},
  publisher={IEEE}
}

@article{kuang2019pricing,
  title={Pricing in non-convex markets with quadratic deliverability costs},
  author={Kuang, Xiaolong and Lamadrid, Alberto J and Zuluaga, Luis F},
  journal={Energy Economics},
  doi = {https://doi.org/10.1016/j.eneco.2018.12.022},
  volume={80},
  pages={123--131},
  year={2019},
  publisher={Elsevier}
}

@article{vlachos2016comparison,
  title={{Comparison of two mathematical programming models for the solution of a convex portfolio-based European day-ahead electricity market}},
  author={Vlachos, AG and Dourbois, GA and Biskas, PN},
  journal={Electric Power Systems Research},
  doi = {https://doi.org/10.1016/j.epsr.2016.08.007},
  volume={141},
  pages={313--324},
  year={2016},
  publisher={Elsevier}
}

@article{mays2021investment,
  title={Investment effects of pricing schemes for non-convex markets},
  author={Mays, Jacob and Morton, David P and O’Neill, Richard P},
  journal={European Journal of Operational Research},
  doi = {https://doi.org/10.1016/j.ejor.2020.07.026},
  volume={289},
  number={2},
  pages={712--726},
  year={2021},
  publisher={Elsevier}
}

@article{knueven2022computationally,
  title={A computationally efficient algorithm for computing convex hull prices},
  author={Knueven, Bernard and Ostrowski, James and Castillo, Anya and Watson, Jean-Paul},
  journal={Computers $\&$ Industrial Engineering},
  doi = {https://doi.org/10.1016/j.cie.2021.107806},
  volume={163},
  pages={107806},
  year={2022},
  publisher={Elsevier}
}

@article{sleisz2019new,
  title={{New formulation of power plants’ general complex orders on European electricity markets}},
  author={Sleisz, {\'A}d{\'a}m and Div{\'e}nyi, D{\'a}niel and Raisz, D{\'a}vid},
  journal={Electric Power Systems Research},
  doi = {https://doi.org/10.1016/j.epsr.2018.12.028},
  volume={169},
  pages={229--240},
  year={2019},
  publisher={Elsevier}
}

@article{martin2014strict,
  title={{Strict linear prices in non-convex European day-ahead electricity markets}},
  author={Martin, Alexander and M{\"u}ller, Johannes C and Pokutta, Sebastian},
  journal={Optimization Methods and Software},
  doi = {https://doi.org/10.1080/10556788.2013.823544},
  volume={29},
  number={1},
  pages={189--221},
  year={2014},
  publisher={Taylor \& Francis}
}

@article{ilea2018european,
  title={European day-ahead electricity market coupling: Discussion, modeling, and case study},
  author={Ilea, Valentin and Bovo, Cristian and others},
  journal={Electric Power Systems Research},
  doi = {https://doi.org/10.1016/j.epsr.2017.10.003},
  volume={155},
  pages={80--92},
  year={2018},
  publisher={Elsevier}
}

@article{madani2017mip,
  title={{A MIP framework for non-convex uniform price day-ahead electricity auctions}},
  author={Madani, Mehdi and Van Vyve, Mathieu},
  journal={EURO Journal on Computational Optimization},
  doi = {https://doi.org/10.1007/s13675-015-0047-6},
  volume={5},
  number={1},
  pages={263--284},
  year={2017},
  publisher={Springer}
}

@article{stevens2024some,
  title={{On some advantages of convex hull pricing for the European electricity auction}},
  author={Stevens, Nicolas and Papavasiliou, Anthony and Smeers, Yves},
  journal={Energy Economics},
  doi = {https://doi.org/10.1016/j.eneco.2024.107542},
  volume={134},
  pages={107542},
  year={2024},
  publisher={Elsevier}
}

@article{madani2018revisiting,
  title={Revisiting minimum profit conditions in uniform price day-ahead electricity auctions},
  author={Madani, Mehdi and Van Vyve, Mathieu},
  journal={European Journal of Operational Research},
  doi = {https://doi.org/10.1016/j.ejor.2017.10.024},
  volume={266},
  number={3},
  pages={1072--1085},
  year={2018},
  publisher={Elsevier}
}

@article{liberopoulos2016critical,
  title={Critical review of pricing schemes in markets with non-convex costs},
  author={Liberopoulos, George and Andrianesis, Panagiotis},
  journal={Operations Research},
  doi = {https://doi.org/10.1287/opre.2015.1451},
  volume={64},
  number={1},
  pages={17--31},
  year={2016},
  publisher={INFORMS}
}

@article{wang2013subgradient,
  title={The subgradient simplex cutting plane method for extended locational marginal prices},
  author={Wang, Congcong and Peng, Tengshun and Luh, Peter B and Gribik, Paul and Zhang, Li},
  journal={IEEE Transactions on Power Systems},
  doi = {10.1109/TPWRS.2013.2243173},
  volume={28},
  number={3},
  pages={2758--2767},
  year={2013},
  publisher={IEEE}
}

@article{yu2020extended,
  title={An extended integral unit commitment formulation and an iterative algorithm for convex hull pricing},
  author={Yu, Yanan and Guan, Yongpei and Chen, Yonghong},
  journal={IEEE Transactions on Power Systems},
  doi = {10.1109/TPWRS.2020.2993027},
  volume={35},
  number={6},
  pages={4335--4346},
  year={2020},
  publisher={IEEE}
}

@article{hua2016convex,
  title={A convex primal formulation for convex hull pricing},
  author={Hua, Bowen and Baldick, Ross},
  journal={IEEE Transactions on Power Systems},
    doi={10.1109/TPWRS.2016.2637718},
  volume={32},
  number={5},
  pages={3814--3823},
  year={2016},
  publisher={IEEE}
}

@book{wood2013power,
  title={Power generation, operation, and control},
  author={Wood, Allen J and Wollenberg, Bruce F and Shebl{\'e}, Gerald B},
  isbn= {0002011697},
  year={2013},
  publisher={John wiley \& sons}
}

@inproceedings{milgrom2022linear,
  title={Linear pricing mechanisms for markets without convexity},
  doi = {https://doi.org/10.1145/3490486.3538310},
  author={Milgrom, Paul and Watt, Mitchell},
  booktitle={Proceedings of the 23rd ACM Conference on Economics and Computation},
  pages={300--300},
  year={2022}
}
\newpage
\section*{Supplementary Material}
\section*{Introduction}
This document provides supplementary proofs and illustrative examples to further explain the lemmas and theorems presented in Section 3 of the paper.
\section*{Proof of [Lemma 1, Section 3.1]
}\label{proof:lemma:has:subgradeint}
\begin{proof}
\par For the variable $\lambda\in \mathbb{R}^{|\mathcal{A}|\times |\mathcal{T}|}$, the inner-level optimization problem is a minimization problem of a set of linear functions with respect to $\lambda$. Therefore, the objective function $z_\rho^{ALD}(\lambda)$ of the dual problem is piecewise linear, and due to that, given a vector $\lambda_1\in \mathbb{R}^{|\mathcal{A}|\times |\mathcal{T}|}$, the gradients of $z_\rho^{ALD}(\lambda_1)$ can be different from different directions. For a one-dimensional case, the derivative of $z_\rho^{ALD}(\lambda_1)$ from the left hand side can be different from that of the right hand side. Therefore, $z_\rho^{ALD}(\lambda)$ is non-differentiable.
\par 
The dual variable $\lambda\in \mathbb{R}^{|\mathcal{A}|\times |\mathcal{T}|}$, and therefore the feasible domain is convex. Given $\lambda_1\in \mathbb{R}^{|\mathcal{A}|\times |\mathcal{T}|},\lambda_2\in \mathbb{R}^{|\mathcal{A}|\times |\mathcal{T}|},\alpha \in [0,1]$, we have
\begin{align*}
    &z_\rho^{ALD}(\alpha\lambda_1 + (1-\alpha) \lambda_2)\\
    &=\min_{x\in X,p\in P} {L_\rho(x,p,\alpha\lambda_1 + (1-\alpha) \lambda_2)}\\
    &\geq \alpha \min_{x\in X,p\in P} {L_\rho(x,p,\lambda_1)} +(1-\alpha) \min_{x\in X,p\in P} {L_\rho(x,p,\lambda_2)} \\
    & = \alpha z_\rho^{ALD}(\lambda_1) + (1-\alpha) z_\rho^{ALD}(\lambda_2).
\end{align*}
Hence, the dual optimization problem is concave. Since the objective function  $z_\rho^{ALD}(\lambda)$ is concave and piecewise linear, it is sub-differentiable. Additionally, the subgradient of the dual objective function $z_\rho^{ALD}(\lambda) $ at a vector ${\lambda^* \in \mathbb{R}^{|\mathcal{A}|\times |\mathcal{T}|}}$ denoted as  $\partial z_\rho^{ALD}(\lambda^*) $ satisfies,
\begin{align*}
z_\rho^{ALD}(\lambda)\leq z_\rho^{ALD}(\lambda^*) + \partial z_\rho^{ALD}(\lambda^*) ^\top (\lambda -\lambda^*).
\end{align*}
\end{proof}
\section*{Example 1}\label{appendix:1}
The examples in this section complement [Corollary 1, Theorem 4, Section 3.2.2]. They serve to demonstrate the possibility of negative surplus of demand orders under the ALD pricing mechanism, validate the proposed remedy, and illustrate the resulting congestion price signals. All prices are expressed in EUR/MWh. The cost, revenue, surplus, etc are in EUR.
\par 
\textbf{1)} Consider the following optimization problem.
\begin{equation}\label{counter:example:1}
\begin{aligned}
    &\min{-23 y_1  -10 y_2-16 y_3+ 14 x_1 + 20 x_2 },\\
    &\text{or equivalently}\\
     &\max{23 y_1  +10 y_2+16 y_3- 14 x_1 - 20 x_2 },\\
    &s.t. ~~  p_1= x_1 - y_3, p_2 = x_2 - y_1 -y_2,\\
    &\qquad f = p_1, p_2 = - p_1, f\in [-100, 100],\\
    &\qquad y_1\in [0,100],
     y_2\in [0,20],
     y_3\in \{0,50\},
     x_1 \in \{0,120\},
    x_2 \in \{0,50\}.
    \end{aligned}
\end{equation}
In Problem \ref{counter:example:1}, we consider a two-area market consisting of two supply block orders ($x_1, x_2$), one demand block order ($y_3$), and two demand elementary orders ($y_1, y_2$). The cost minimization problem yields an optimal objective value of $-730$ (corresponding to a welfare of $730$ EUR), with the optimal solution: $x_1^* = 120, x_2^* = 0, y_1^* = 70, y_2^* = 0, y_3^* = 50$, and a cross-zonal flow $f^* = 70$. 
\par
Under the ALD approach, the dual solution is characterized by $\lambda_1 \in [14, 14.5]$ and $\lambda_2 \in [14, 20]$ under the penalty coefficient threshold $\rho^* = 9$ required to close the duality gap. Applying the ALD pricing mechanism [Def. 1, Section 3.2], the resulting price signals are $LMP_1 = 14$ (EUR/MWh) and $LMP_2 = 14$ (EUR/MWh). Consequently, the demand order $y_3$ has a commodity cost of $14 \times 50$ (EUR) and a non-convex charge of $9 \times 50$ (EUR), leading to a surplus of $16 \times 50 - (14 \times 50 + 9 \times 50) = -350$ (EUR). This negative surplus identifies $y_3$ as a Paradoxically Accepted Order (PAO), which is due to the fact that certain demand bid prices are lower than the supply offer prices.
\par
\textbf{Remedial Approach}: For the demand order $y_3$, if a transfer payment of $2 \times 9 \times 50 = 900$ (EUR) is applied, the adjusted surplus becomes $(16 + 9 - 14) \times 50 = 550$ (EUR). This confirms that the proposed remedy restores the individual revenue adequacy of the paradoxically accepted orders among demand orders.
\par
\textbf{Congestion Price Signals}: Furthermore, this example reveals that the raw ALD pricing mechanism may fail to yield a congestion price signal despite the presence of network congestion. Note that the cross-zonal price difference by ALD pricing is $LMP_1 - LMP_2 = 14 -14 = 0$ (EUR/MWh). For comparison, the LP relaxation provides optimal dual variables $\lambda_1 = 16$ and $\lambda_2 \in [16, 20]$. By selecting $LMP_1 = 16$ (EUR/MWh) and $LMP_2 = 17$ (EUR/MWh), the resulting congestion price for the TSO is $1$ (EUR/MWh). Note that the price signals provided by LP relaxation also effectively eliminate both PAOs and PROs.
\par
This case highlights a critical limitation: raw ALD dual variables may fail to internalize network congestion even when flow limits are active. To address this, the introduction of the parameter $\eta$ (as proposed in [Definition 1, Section 3.2]) provides the necessary flexibility to restore congestion price signals. This ensures revenue adequacy for the TSO, particularly in scenarios of network congestion, while maintaining the theoretical advantages of the ALD framework. Note that by adding $\eta = 1$, the ALD pricing mechanism gives $LMP_1 = 14$ (EUR/MWh) and $LMP_2 = 15$ (EUR/MWh), implying a congestion price of 1 EUR/MWh for the TSO.
\par
\textbf{2) Scenario with excluded non-competitive bids:} When demand orders with bid prices below supply offers are excluded, the market-clearing problem is formulated as follows:
\begin{equation}\label{counter:example:2}
\begin{aligned}&\max \quad 23 y_1 + 20 y_2 + 16 y_3 - 14 x_1 - 10 x_2 \\
&\text{s.t.} \quad p_1 = x_1 - y_3, \quad p_2 = x_2 - y_1 - y_2 \\
&\qquad f = p_1, \quad p_2 = - p_1, \quad f \in [-100, 100] \\
&\qquad y_1 \in [0, 100], \quad y_2 \in [0, 20], \quad y_3 \in \{0, 50\}\\
&\qquad x_1 \in \{0, 120\}, \quad x_2 \in \{0, 50\}
\end{aligned}
\end{equation}
The optimal solution is $x_1^* = 120, x_2^* = 50, y_1^* = 100, y_2^* = 20, y_3^* = 50$, indicating full acceptance of all orders with a cross-zonal flow $f^* = 70$. Under the ALD framework, the solution is characterized by $\lambda_1 = 13$ and $\lambda_2 \in [13, 15]$. Setting $\rho = 1$ (where $\rho > \rho^* = 0$) and applying the mechanism in [Def. 1, Section 3.2], the prices are determined as $LMP_1 = 13$ (EUR/MWh) and $LMP_2 = 14$ (EUR/MWh). This yields $\eta = 1$ (EUR/MWh), representing the congestion price for the TSO. Furthermore, all orders satisfy the revenue adequacy requirement. In contrast, while LP relaxation-based prices (e.g., $LMP_1 = LMP_2 = 14$ or $15$ (EUR/MWh)) also eliminate PAOs and PROs, they do not generate a congestion price for the TSO and thus fail to represent network congestion.
 \section*{Case Study: Dependency and Dispatch Restrictions}\label{appendix:2}This example illustrates a scenario where the market-clearing price exceeds a specific supply order's offer price, yet the order's dispatch remains at zero due to feasibility constraints. Consider the following problem:
 \begin{equation}\label{a:LP}
 \begin{aligned}
 \min \quad & -50 y_1 - 35 y_2 + 40 x_1 + 30 x_2 + 32 x_3 \\
 \text{s.t.} \quad & x_1 + x_2 = y_1 + y_2, \quad (\lambda) \\
 & x_3 \leq x_1, \quad (\xi) \\& y_1 \in [0, 50], \quad y_2 \in [0, 50], \quad x_1 \in [0, 30] \\
 & x_2 \in [0, 50], \quad x_3 \in [0, 30]
 \end{aligned}
 \end{equation}
 Obviously, Problem \eqref{a:LP} is an LP, where there are two demand orders ($y_1, y_2$) and three supply orders ($x_1, x_2, x_3$). The optimal solution is $y_1^* = 50, y_2^* = 0, x_1^* = 0, x_2^* = 50, x_3^* = 0$, yielding the optimal objective value of $-1000$. Any $\lambda^* \in [35, 40]$ constitutes an optimal dual solution. Although $\lambda^* > 32$ (the offer price of $x_3$), the supply order $x_3$ is not classified as a Paradoxically Rejected Order (PRO). This is because its feasibility is dependent on the dispatch of its parent order ($x_1$). Following [Definition 2, Section 3.2], we first evaluate the individual optimization of the parent order $x_1$:\begin{equation}\label{indivi:x:1}\max \quad (\lambda^* - 40) x_1, \quad \text{s.t.} \quad x_1 \in [0, 30]\end{equation}Since $\lambda^* \leq 40$, $x_1^* = 0$ is an optimal solution. Subsequently, the individual optimization for the child order $x_3$ is formulated as:\begin{equation}\label{indivi:x:3}\max \quad (\lambda^* - 32) x_3, \quad \text{s.t.} \quad 0 \leq x_3 \leq x_1^*\end{equation}Given $x_1^* = 0$, the only feasible solution for \eqref{indivi:x:3} is $x_3^* = 0$. Consequently, $x_3$ is not a PRO as its zero-dispatch is a requirement of feasibility rather than a rejection of a profitable bid.
  \section*{Case Study: Elimination of PROs and PAOs}\label{appendix:3}This example demonstrates how the ALD pricing mechanism eliminates Paradoxically Rejected Orders (PROs) and Paradoxically Accepted Orders (PAOs) in a Unit Commitment (UC) framework. Consider the following optimization problem (Pure Integer Programming):
  \begin{equation}\label{an:example:to:eliminate:PAO:PRO}
  \begin{aligned}\min \quad & 100 x_1 + 315 x_2 + 45 x_3 \\
  \text{s.t.} \quad & 4 x_1 + 9 x_2 + x_3 = 10, \quad x_i \in \{0, 1\}, i= 1, 2, 3.
  \end{aligned}
  \end{equation}
  Three supply block orders produce electricity to meet a 10 MWh demand. The offer prices are 25, 35, and 45 EUR/MWh, with capacities of 4, 9, and 1 MWh, respectively. The optimal solution is $x_1^* = 0, x_2^* = 1, x_3^* = 1$, with a total cost of 360 EUR. Under LP relaxation, the LMP is 35 EUR/MWh. At this price, the supply order 1 ($x_1^*=0$) becomes a PRO because its offer (25 EUR/MWh) is below the price, while the supply order 3 ($x_3^*=1$) becomes a PAO because its offer (45 EUR/MWh) exceeds the price. Applying the ALD solution ($\lambda^* = 31.7, \rho = 13.37$), we evaluate the performance of the proposed ALD pricing mechanism as follows:
  \begin{itemize}
  \item \textbf{Supply order 1 ($x_1^* = 0$):} By following the centralized dispatch, its surplus is zero. If it were to deviate to $x_1 = 1$, it would receive a commodity revenue of $31.7 \times 4$ EUR, but incur a cost of $25 \times 4$ EUR, and a non-convex penalty of $13.37 \times 4$ EUR. The resulting surplus, $(31.7 - 25 - 13.37) \times 4 = -26.68$ EUR, which is negative. Thus, the supply order 1 has no incentive to deviate, eliminating the PRO status.
  \item \textbf{Supply order 3 ($x_3^* = 1$):} By following the dispatch, it receives a total payment (commodity revenue plus non-convex reward) of $(31.7 + 13.37) \times 1 = 45.07$ EUR. Deducting the cost ($45 \times 1$) EUR, the surplus of supply order 3 is $0.07$ EUR. Since its surplus is non-negative, the PAO condition is resolved.
  \end{itemize}
  Furthermore, it can be verified that under the ALD pricing mechanism, supply order 2 has a non-negative surplus and has no incentive to deviate from $x_2^* = 1$; supply order 3 has no incentive to deviate from $x_3^* = 1$. Consequently, the ALD mechanism ensures individual revenue adequacy for supply orders and incentive compatibility for each order.

\end{document}